\documentclass[reqno]{amsart}
\usepackage{amsfonts}

\newtheorem{theorem}{Theorem}
\theoremstyle{plain}

\newtheorem{corollary}{Corollary}

\newtheorem{example}{Example}

\newtheorem{lemma}{Lemma}

\newtheorem{proposition}{Proposition}

\numberwithin{equation}{section}
\numberwithin{theorem}{section}
\numberwithin{algorithm}{section}
\numberwithin{axiom}{section}
\numberwithin{case}{section}
\numberwithin{claim}{section}
\numberwithin{conclusion}{section}
\numberwithin{condition}{section}
\numberwithin{conjecture}{section}
\numberwithin{corollary}{section}
\numberwithin{criterion}{section}
\numberwithin{definition}{section}
\numberwithin{example}{section}
\numberwithin{exercise}{section}
\numberwithin{lemma}{section}
\numberwithin{notation}{section}
\numberwithin{problem}{section}
\numberwithin{proposition}{section}
\numberwithin{remark}{section}
\numberwithin{solution}{section}

\input{tcilatex}

\begin{document}
\title{Paneitz operator for metrics near $S^{3}$}
\author{Fengbo Hang}
\address{Courant Institute, New York University, 251 Mercer Street, New York
NY 10012}
\email{fengbo@cims.nyu.edu}
\author{Paul C. Yang}
\address{Department of Mathematics, Princeton University, Fine Hall,
Washington Road, Princeton NJ 08544}
\email{yang@math.princeton.edu}
\date{}

\begin{abstract}
We derive the first and second variation formula for the Green's function
pole's value of Paneitz operator on the standard three sphere. In particular
it is shown that the first variation vanishes and the second variation is
nonpositively definite. Moreover, the second variation vanishes only at the
direction of conformal deformation. We also introduce a new invariant of the
Paneitz operator and illustrate its close relation with the second
eigenvalue and Sobolev inequality of Paneitz operator.
\end{abstract}

\maketitle

\section{Introduction\label{sec1}}

The fourth order $Q$ curvature equation (\cite{Br,P}) has attracted interest
due to the successful study in dimension four and its application to
conformal geometry in dimension four (\cite{CGY}). We are interested to
possibly extend this analysis to dimension three. The effort to understand $%
Q $ curvature in dimension three motivates many intriguing and challenging
problems. In this dimension, the functions in $H^{2}$ are actually $\frac{1}{%
2}$-Holder continuous, and hence the Green's function has well defined value
at its pole. The sign of this value turns out to be an important issue. For
the standard sphere, the Green's function is nonpositive everywhere but
vanishes exactly at the pole. Our purpose in this article is to study this
question for the conformal structures near the standard sphere.

Recall on a three manifold, the $Q$ curvature is given by%
\begin{equation}
Q=-\frac{1}{4}\Delta R-2\left\vert Rc\right\vert ^{2}+\frac{23}{32}R^{2},
\label{eq1.1}
\end{equation}%
and the fourth order Paneitz operator is defined as%
\begin{equation}
P\varphi =\Delta ^{2}\varphi +4\func{div}\left( Rc\left( \nabla \varphi
,e_{i}\right) e_{i}\right) -\frac{5}{4}\func{div}\left( R\nabla \varphi
\right) -\frac{1}{2}Q\varphi .  \label{eq1.2}
\end{equation}%
Here $Rc$ is the Ricci curvature, $R$ is the scalar curvature and $%
e_{1},e_{2},e_{3}$ is a local orthonormal frame with respect to the metric.
For any positive smooth function $\rho $, the operator satisfies%
\begin{equation}
P_{\rho ^{-4}g}\varphi =\rho ^{7}P_{g}\left( \rho \varphi \right) .
\label{eq1.3}
\end{equation}%
As a consequence we know $\ker P_{g}=0\Leftrightarrow \ker P_{\rho ^{-4}g}=0$
and under this assumption, the Green's functions of $P$ satisfy the
transformation law%
\begin{equation}
G_{\rho ^{-4}g}\left( p,q\right) =\rho \left( p\right) ^{-1}\rho \left(
q\right) ^{-1}G_{g}\left( p,q\right) .  \label{eq1.4}
\end{equation}

Let $\left( M,g\right) $ be a smooth compact three dimensional Riemannian
manifold, for $u,v\in C^{\infty }\left( M\right) $, we write%
\begin{eqnarray}
E\left( u,v\right) &=&\int_{M}Pu\cdot vd\mu  \label{eq1.5} \\
&=&\int_{M}\left( \Delta u\Delta v-4Rc\left( \nabla u,\nabla v\right) +\frac{%
5}{4}R\nabla u\cdot \nabla v-\frac{1}{2}Quv\right) d\mu  \notag
\end{eqnarray}%
and $E\left( u\right) =E\left( u,u\right) $. Here $\mu $ is the measure
associated with metric $g$. It is clear that $E\left( u,v\right) $ makes
sense for $u,v\in H^{2}\left( M\right) $.

The scaling invariant quantity $\mu \left( M\right) ^{\frac{1}{3}%
}\int_{M}Qd\mu $ satisfies%
\begin{equation*}
\mu _{\rho ^{-4}g}\left( M\right) ^{\frac{1}{3}}\int_{M}Q_{\rho ^{-4}g}d\mu
_{\rho ^{-4}g}=-2E_{g}\left( \rho \right) \left\Vert \rho ^{-1}\right\Vert
_{L^{6}\left( M,g\right) }^{2}.
\end{equation*}%
Hence%
\begin{eqnarray*}
\sup_{\widetilde{g}\in \left[ g\right] }\widetilde{\mu }\left( M\right) ^{%
\frac{1}{3}}\int_{M}\widetilde{Q}d\widetilde{\mu } &=&-2\inf_{\rho \in
C^{\infty },\rho >0}E\left( \rho \right) \left\Vert \rho ^{-1}\right\Vert
_{L^{6}}^{2} \\
&=&-2\inf_{u\in H^{2}\left( M\right) ,u>0}E\left( u\right) \left\Vert
u^{-1}\right\Vert _{L^{6}}^{2}.
\end{eqnarray*}%
Here $\left[ g\right] $ is the conformal class of metrics associated with $g$%
. As in \cite{HY1}, we write%
\begin{equation}
I_{4}\left( u\right) =E\left( u\right) \left\Vert u^{-1}\right\Vert
_{L^{6}}^{2}  \label{eq1.6}
\end{equation}%
and%
\begin{equation}
Y_{4}\left( g\right) =\inf_{u\in H^{2}\left( M\right) ,u>0}E\left( u\right)
\left\Vert u^{-1}\right\Vert _{L^{6}}^{2}.  \label{eq1.7}
\end{equation}%
From above discussion we see $Y_{4}\left( \widetilde{g}\right) =Y_{4}\left(
g\right) $ for $\widetilde{g}\in \left[ g\right] $.

The question of whether $Y_{4}\left( g\right) $ is finite and achieved by
some particular metrics was considered in \cite{HY1,YZ}. This inequality is
analytically different from the one of Yamabe invariant $Y\left( g\right) $
(see \cite{LP}) due to the negative power involved.

\cite{HY1} shows that when $\ker P=0$, the value of the Green's function at
pole plays a crucial role. In particular based on explicit calculation of
this value on Berger's sphere, we were able to show $Y_{4}\left( g\right) $
is achieved on all Berger spheres. In general such an explicit formula is
not available. On the other hand, properties of Paneitz operator on the
standard three sphere are well understood.

On standard $S^{3}$, we have%
\begin{equation}
Pu=\Delta ^{2}u+\frac{1}{2}\Delta u-\frac{15}{16}u.  \label{eq1.8}
\end{equation}%
Let $N$ be the north pole, $\pi _{N}:S^{3}\backslash \left\{ N\right\}
\rightarrow \mathbb{R}^{3}$ be the stereographic projection, using $x=\pi
_{N}$ as the coordinates, the Green's function of $P$ with pole at $N$ is
given by%
\begin{equation}
G_{N}=-\frac{1}{4\pi }\frac{1}{\sqrt{\left\vert x\right\vert ^{2}+1}}.
\label{eq1.9}
\end{equation}%
In particular $G_{N}\left( N\right) =0$.

\begin{proposition}
\label{prop1.1}Let $g$ be the standard metric on $S^{3}$, then for any $p\in
S^{3}$ and any smooth symmetric $\left( 0,2\right) $ tensor $h$,%
\begin{equation*}
\left. \partial _{t}\right\vert _{t=0}G_{g+th}\left( p,p\right) =0.
\end{equation*}%
Here $G_{g+th}$ is the Green's function of the Paneitz operator $P_{g+th}$.
\end{proposition}

This calculation leads one to ask about the second variation. We have

\begin{theorem}
\label{thm1.1}Using the stereographic projection $\pi _{N}$ as the
coordinate, the standard metric $g$ on $S^{3}$ is written as%
\begin{equation}
g=\frac{4}{\left( \left\vert x\right\vert ^{2}+1\right) ^{2}}\left\vert
dx\right\vert ^{2}=\tau ^{-4}\left\vert dx\right\vert ^{2},  \label{eq1.10}
\end{equation}%
here%
\begin{equation}
\tau =\sqrt{\frac{\left\vert x\right\vert ^{2}+1}{2}}.  \label{eq1.11}
\end{equation}%
For any smooth symmetric $\left( 0,2\right) $ tensor $h$, denote%
\begin{equation}
\theta =\tau ^{4}h,  \label{eq1.12}
\end{equation}%
then we have%
\begin{eqnarray}
&&\left. \partial _{t}^{2}\right\vert _{t=0}G_{g+th}\left( N,N\right)
\label{eq1.13} \\
&=&-\frac{1}{64\pi ^{2}}\int_{\mathbb{R}^{3}}\left( \sum_{ij}\left( \theta
_{ikjk}+\theta _{jkik}-\left( tr\theta \right) _{ij}-\Delta \theta
_{ij}\right) ^{2}-\frac{3}{2}\left( \theta _{ijij}-\Delta tr\theta \right)
^{2}\right) dx.  \notag
\end{eqnarray}%
Here the derivatives $\theta _{ikjk}$ etc are taken with respect to the
standard metric on $\mathbb{R}^{3}$.
\end{theorem}

Using formula (\ref{eq1.13}), we will show the second variation is always
nonpositive and it vanishes only in the direction generated by conformal
diffeomorphism. More precisely,

\begin{proposition}
\label{prop1.2}For any smooth symmetric $\left( 0,2\right) $ tensor $h$ on $%
S^{3}$ and $p\in S^{3}$,%
\begin{equation}
\left. \partial _{t}^{2}\right\vert _{t=0}G_{g+th}\left( p,p\right) \leq 0
\label{eq1.14}
\end{equation}%
Moreover, $\left. \partial _{t}^{2}\right\vert _{t=0}G_{g+th}\left(
p,p\right) =0$ if and only if $h=L_{X}g+f\cdot g$ for some smooth vector
fields $X$ and smooth function $f$ on $S^{3}$.
\end{proposition}

It is worth pointing out that in \cite{HY2}, motivated from recent works 
\cite{GM,HR} for $Q$ curvature in dimension five or higher and Proposition %
\ref{prop1.1} and \ref{prop1.2} above, it was shown that for smooth compact
three manifold $\left( M,g\right) $ with positive scalar and $Q$ curvature,
the Paneitz operator must have zero kernel and the Green's function pole's
value is strictly negative except when $\left( M,g\right) $ is conformal
diffeomorphic to the standard $S^{3}$. Further developments can be found in 
\cite{HY3,HY4}.

The Sobolev inequality of Paneitz operator on $S^{3}$ was first verified in 
\cite{YZ}. Different proofs were given in \cite{H,HY1}. The new approach
motivates the condition NN and condition P for a Paneitz operator (see \cite[%
section 5]{HY1}). Here we will introduce a quantity for the Paneitz operator
whose sign corresponds to condition NN and P. Let $\left( M,g\right) $ be a
smooth compact three dimensional Riemannian manifold without boundary. For
any $p\in M$, denote (recall functions in $H^{2}\left( M\right) $ are $\frac{%
1}{2}$-Holder continuous)%
\begin{equation}
\nu \left( M,g,p\right) =\inf \left\{ \frac{E\left( u\right) }{%
\int_{M}u^{2}d\mu }:u\in H^{2}\left( M\right) \backslash \left\{ 0\right\}
,u\left( p\right) =0\right\} .  \label{eq1.15}
\end{equation}%
When no confusion could arise we denote it as $\nu \left( g,p\right) $ or $%
\nu _{p}$. We also write%
\begin{eqnarray}
\nu \left( M,g\right) &=&\inf_{p\in M}\nu \left( M,g,p\right)  \label{eq1.16}
\\
&=&\inf \left\{ \frac{E\left( u\right) }{\int_{M}u^{2}d\mu }:u\in
H^{2}\left( M\right) \backslash \left\{ 0\right\} ,u\left( p\right) =0\text{
for some }p\right\} .  \notag
\end{eqnarray}%
The importance of $\nu \left( M,g\right) $ lies in that $\left( M,g\right) $
satisfies condition P if and only if $\nu \left( M,g\right) >0$ and it
satisfies condition NN if and only if $\nu \left( M,g\right) \geq 0$. For
the standard metric $g$ on $S^{3}$, $\nu \left( S^{3},g\right) =0$, in fact
we have (see Example \ref{ex4.1}) $\nu \left( S^{3},g,p\right) =0$ for all $%
p\in S^{3}$.

\begin{theorem}
\label{thm1.2}Let $g$ be the standard metric on $S^{3}$, then for any $p\in
S^{3}$ and any smooth symmetric $\left( 0,2\right) $ tensor $h$,%
\begin{equation}
\left. \partial _{t}\right\vert _{t=0}\nu \left( g+th,p\right) =0
\label{eq1.17}
\end{equation}%
and%
\begin{equation}
\left. \partial _{t}^{2}\right\vert _{t=0}\nu \left( g+th,p\right)
=-16\left. \partial _{t}^{2}\right\vert _{t=0}G_{g+th}\left( p,p\right) .
\label{eq1.18}
\end{equation}%
In particular, $\left. \partial _{t}^{2}\right\vert _{t=0}\nu \left(
g+th,p\right) \geq 0$ and it vanishes if and only if $h=L_{X}g+f\cdot g$ for
some smooth vector fields $X$ and smooth functions $f$ on $S^{3}$.
\end{theorem}

Roughly speaking Theorem \ref{thm1.2} tells us for Riemannian metrics near
the standard metric on $S^{3}$, as long as it is not conformal diffeomorphic
to the standard sphere, condition P is satisfied and hence $Y_{4}\left(
g\right) $ is achieved by \cite{HY1}. More related results can be found in 
\cite{HY2,HY3}.

In section \ref{sec2} below we will introduce technique simplifying various
calculations. In section \ref{sec3} we will derive the first and second
variation formulas and justify its nonpositivity. In section \ref{sec4}, we
will study the quantity $\nu \left( g\right) $ and its relations to $%
Y_{4}\left( g\right) $ and the second eigenvalue of Paneitz operator. Some
of the lengthy calculations are collected in the appendix to streamline the
discussions.

\section{Some preparations\label{sec2}}

Because the formula of $Q$ curvature and Paneitz operator are relatively
complicated, it is crucial to take advantage of the conformal covariant
property (\ref{eq1.3}) to simplify the calculation of first and second
variation of the Green's function pole's value. To achieve this we observe
that the Paneitz operator gives us a sequence of fourth order conformal
covariant operators. Indeed for smooth metric $g$ and symmetric $\left(
0,2\right) $ tensor $h$, we define the operator $P_{g,h}^{\left( k\right) }$
by the Taylor expansion%
\begin{equation}
P_{g+th}\varphi \sim \sum_{k=0}^{\infty }t^{k}P_{g,h}^{\left( k\right)
}\varphi .  \label{eq2.1}
\end{equation}%
Here $\sim $ means for any $m\geq 0$,%
\begin{equation*}
P_{g+th}\varphi =\sum_{k=0}^{m}t^{k}P_{g,h}^{\left( k\right) }\varphi
+O\left( t^{m+1}\right)
\end{equation*}%
as $t\rightarrow 0$.

\begin{lemma}
\label{lem2.1}For any smooth function $\varphi $ and positive smooth
function $\rho $,%
\begin{equation}
P_{\rho ^{-4}g,\rho ^{-4}h}^{\left( k\right) }\varphi =\rho
^{7}P_{g,h}^{\left( k\right) }\left( \rho \varphi \right) .  \label{eq2.2}
\end{equation}
\end{lemma}

This is the conformal covariant property of $P_{g,h}^{\left( k\right) }$.
Indeed for $t$ near $0$, 
\begin{equation*}
P_{\rho ^{-4}\left( g+th\right) }\varphi =\rho ^{7}P_{g+th}\left( \rho
\varphi \right) =P_{\rho ^{-4}g+t\rho ^{-4}h}\varphi ,
\end{equation*}%
hence%
\begin{equation*}
\sum_{k=0}^{\infty }t^{k}P_{\rho ^{-4}g,\rho ^{-4}h}^{\left( k\right)
}\varphi \sim \rho ^{7}\sum_{k=0}^{\infty }t^{k}P_{g,h}^{\left( k\right)
}\left( \rho \varphi \right) .
\end{equation*}%
Equation (\ref{eq2.2}) follows.

Careful calculation shows (see appendix)%
\begin{eqnarray}
&&P_{g,h}^{\left( 1\right) }\varphi  \label{eq2.3} \\
&=&-h_{ij}\left( \Delta \varphi \right) _{ij}-\Delta \left( h_{ij}\varphi
_{ij}\right) -\frac{1}{2}\left( 2h_{ijj}-\left( trh\right) _{i}\right)
\left( \Delta \varphi \right) _{i}-\frac{1}{2}\Delta \left( \left(
2h_{ijj}-\left( trh\right) _{i}\right) \varphi _{i}\right)  \notag \\
&&+2\left( 2h_{ikjk}-\Delta h_{ij}-\left( trh\right) _{ij}\right) \varphi
_{ij}-8Rc_{ij}h_{ik}\varphi _{jk}+\frac{5}{4}Rh_{ij}\varphi _{ij}  \notag \\
&&-\frac{5}{4}\left( h_{ijij}-\Delta trh-Rc_{ij}h_{ij}\right) \Delta \varphi
-2Rc_{ij}\left( 2h_{ikj}-h_{ijk}\right) \varphi _{k}+\frac{5}{8}R\left(
2h_{ijj}-\left( trh\right) _{i}\right) \varphi _{i}  \notag \\
&&+\frac{3}{4}\left( h_{klkl}-\Delta trh-Rc_{kl}h_{kl}\right) _{i}\varphi
_{i}-\frac{3}{4}h_{ij}R_{i}\varphi _{j}+\frac{1}{8}\Delta \left(
h_{ijij}-\Delta trh-Rc_{ij}h_{ij}\right) \varphi  \notag \\
&&-\frac{1}{8}h_{ij}R_{ij}\varphi -\frac{1}{16}\left( 2h_{ijj}-\left(
trh\right) _{i}\right) R_{i}\varphi +Rc_{ij}\left( 2h_{ikjk}-\Delta
h_{ij}-\left( trh\right) _{ij}\right) \varphi  \notag \\
&&-2Rc_{ij}Rc_{ik}h_{jk}\varphi -\frac{23}{32}R\left( h_{ijij}-\Delta
trh-Rc_{ij}h_{ij}\right) \varphi .  \notag
\end{eqnarray}%
On the other hand the formula of $P_{g,h}^{\left( 2\right) }\varphi $ is
much more complicated, and we will not write it down here. Instead we
observe that $P_{g,h}^{\left( 2\right) }\varphi $ is a fourth order operator
in $\varphi $ and $P_{g,h}^{\left( 2\right) }1$ can be written down in a
reasonable way. Indeed (see appendix)%
\begin{eqnarray}
&&Q_{g+th}  \label{eq2.4} \\
&=&Q-\frac{t}{4}\Delta \left( h_{ijij}-\Delta trh-Rc_{ij}h_{ij}\right)
-2tRc_{ij}\left( 2h_{ikjk}-\left( trh\right) _{ij}-\Delta h_{ij}\right) 
\notag \\
&&+\frac{23t}{16}R\left( h_{ijij}-\Delta trh-Rc_{ij}h_{ij}\right) +\frac{t}{8%
}\left( 2h_{ijj}-\left( trh\right) _{i}\right) R_{i}+\frac{t}{4}%
h_{ij}R_{ij}+4tRc_{ij}Rc_{ik}h_{jk}  \notag \\
&&-\frac{t^{2}}{4}\Delta \left[ \left( \Delta h_{ij}+\left( trh\right)
_{ij}-h_{ikjk}-h_{ikkj}\right) h_{ij}\right] +\frac{t^{2}}{4}h_{ij}\left(
h_{klkl}-\Delta trh-Rc_{kl}h_{kl}\right) _{ij}  \notag \\
&&-\frac{t^{2}}{16}\Delta \sum_{ijk}\left( h_{ikj}+h_{jki}-h_{ijk}\right)
^{2}+\frac{t^{2}}{16}\Delta \sum_{i}\left( 2h_{ijj}-\left( trh\right)
_{i}\right) ^{2}  \notag \\
&&+\frac{t^{2}}{8}\left( 2h_{ijj}-\left( trh\right) _{i}\right) \left(
h_{klkl}-\Delta trh-Rc_{kl}h_{kl}\right) _{i}-\frac{t^{2}}{4}\Delta \left(
Rc_{ij}h_{ij}^{2}\right)  \notag \\
&&-\frac{t^{2}}{2}\sum_{ij}\left( h_{ikjk}+h_{jkik}-\left( trh\right)
_{ij}-\Delta h_{ij}\right) ^{2}+2t^{2}Rc_{ij}h_{kl}\left(
2h_{ikjl}-h_{klij}-h_{ijkl}\right)  \notag \\
&&+4t^{2}Rc_{ij}h_{ik}\left( h_{jlkl}+h_{kljl}-\left( trh\right)
_{jk}-\Delta h_{jk}\right) +\frac{23t^{2}}{32}\left( h_{ijij}-\Delta
trh-Rc_{ij}h_{ij}\right) ^{2}  \notag \\
&&+\frac{23t^{2}}{16}R\left( \Delta h_{ij}+\left( trh\right)
_{ij}-h_{ikjk}-h_{ikkj}\right) h_{ij}-t^{2}Rc_{ij}\left(
h_{ikl}+h_{kli}-h_{ilk}\right) \left( h_{jkl}+h_{klj}-h_{jlk}\right)  \notag
\\
&&+t^{2}Rc_{ij}\left( 2h_{ikj}-h_{ijk}\right) \left( 2h_{kll}-\left(
trh\right) _{k}\right) +\frac{23t^{2}}{64}R\sum_{ijk}\left(
h_{ikj}+h_{jki}-h_{ijk}\right) ^{2}  \notag \\
&&-\frac{23t^{2}}{64}R\sum_{i}\left( 2h_{ijj}-\left( trh\right) _{i}\right)
^{2}-\frac{t^{2}}{8}\left( 2h_{ikk}-\left( trh\right) _{i}\right)
h_{ij}R_{j}-\frac{t^{2}}{8}\left( 2h_{ikj}-h_{ijk}\right) h_{ij}R_{k}  \notag
\\
&&-\frac{t^{2}}{4}%
h_{ij}^{2}R_{ij}-4t^{2}Rc_{ij}Rc_{ik}h_{jk}^{2}-2t^{2}Rc_{ij}Rc_{kl}h_{ik}h_{jl}+%
\frac{23t^{2}}{16}R\cdot Rc_{ij}h_{ij}^{2}+O\left( t^{3}\right) .  \notag
\end{eqnarray}%
Because%
\begin{equation*}
P_{g+th}1=-\frac{1}{2}Q_{g+th},
\end{equation*}%
we deduce that%
\begin{eqnarray}
&&P_{g,h}^{\left( 2\right) }1  \label{eq2.5} \\
&=&-\frac{1}{8}h_{ij}\left( h_{klkl}-\Delta trh-Rc_{kl}h_{kl}\right) _{ij}+%
\frac{1}{8}\Delta \left[ \left( \Delta h_{ij}+\left( trh\right)
_{ij}-h_{ikjk}-h_{ikkj}\right) h_{ij}\right]  \notag \\
&&-\frac{1}{16}\left( 2h_{ijj}-\left( trh\right) _{i}\right) \left(
h_{klkl}-\Delta trh-Rc_{kl}h_{kl}\right) _{i}+\frac{1}{32}\Delta
\sum_{ijk}\left( h_{ikj}+h_{jki}-h_{ijk}\right) ^{2}  \notag \\
&&-\frac{1}{32}\Delta \sum_{i}\left( 2h_{ijj}-\left( trh\right) _{i}\right)
^{2}+\frac{1}{8}\Delta \left( Rc_{ij}h_{ij}^{2}\right) +\frac{1}{4}%
\sum_{ij}\left( h_{ikjk}+h_{jkik}-\left( trh\right) _{ij}-\Delta
h_{ij}\right) ^{2}  \notag \\
&&-Rc_{ij}h_{kl}\left( 2h_{ikjl}-h_{klij}-h_{ijkl}\right)
-2Rc_{ij}h_{ik}\left( h_{jlkl}+h_{kljl}-\left( trh\right) _{jk}-\Delta
h_{jk}\right)  \notag \\
&&-\frac{23}{64}\left( h_{ijij}-\Delta trh-Rc_{ij}h_{ij}\right) ^{2}-\frac{23%
}{32}R\left( \Delta h_{ij}+\left( trh\right) _{ij}-h_{ikjk}-h_{ikkj}\right)
h_{ij}  \notag \\
&&+\frac{1}{2}Rc_{ij}\left( h_{ikl}+h_{kli}-h_{ilk}\right) \left(
h_{jkl}+h_{klj}-h_{jlk}\right) -\frac{1}{2}Rc_{ij}\left(
2h_{ikj}-h_{ijk}\right) \left( 2h_{kll}-\left( trh\right) _{k}\right)  \notag
\\
&&-\frac{23}{128}R\sum_{ijk}\left( h_{ikj}+h_{jki}-h_{ijk}\right) ^{2}+\frac{%
23}{128}R\sum_{i}\left( 2h_{ijj}-\left( trh\right) _{i}\right) ^{2}  \notag
\\
&&+\frac{1}{8}\left[ h_{ij}^{2}R_{ij}+\frac{1}{2}\left( 2h_{ikk}-\left(
trh\right) _{i}\right) h_{ij}R_{j}+\frac{1}{2}\left( 2h_{ikj}-h_{ijk}\right)
h_{ij}R_{k}\right] +2Rc_{ij}Rc_{ik}h_{jk}^{2}  \notag \\
&&+Rc_{ij}Rc_{kl}h_{ik}h_{jl}-\frac{23}{32}R\cdot Rc_{ij}h_{ij}^{2}.  \notag
\end{eqnarray}%
In general, $P_{g,h}^{\left( 1\right) }$ is not self adjoint, instead we have

\begin{lemma}
\label{lem2.2}For every $\varphi ,\psi \in C^{\infty }$,%
\begin{equation}
\int_{M}P_{g,h}^{\left( 1\right) }\varphi \cdot \psi d\mu =\int_{M}\varphi
P_{g,h}^{\left( 1\right) }\psi d\mu -\frac{1}{2}\int_{M}\left( P\varphi
\cdot \psi -\varphi P\psi \right) trhd\mu .  \label{eq2.6}
\end{equation}
\end{lemma}

Indeed this follows from the Taylor expansion in $t$ for%
\begin{equation*}
\int_{M}P_{g+th}\varphi \cdot \psi d\mu _{g+th}=\int_{M}\varphi P_{g+th}\psi
d\mu _{g+th}.
\end{equation*}

To derive a variational formula for the Green's function pole's value we
write%
\begin{equation}
G_{g+th}\left( p,q\right) =G\left( p,q\right) +tI\left( p,q,h\right)
+t^{2}II\left( p,q,h\right) +O\left( t^{3}\right) .  \label{eq2.7}
\end{equation}%
Note that%
\begin{equation}
\left. \partial _{t}\right\vert _{t=0}G_{g+th}\left( p,p\right) =I\left(
p,p,h\right)  \label{eq2.8}
\end{equation}%
and%
\begin{equation}
\left. \partial _{t}^{2}\right\vert _{t=0}G_{g+th}\left( p,p\right)
=2II\left( p,p,h\right) .  \label{eq2.9}
\end{equation}%
We can write $I$ and $II$ in terms of $P_{g,h}^{\left( 1\right) }$ and $%
P_{g,h}^{\left( 2\right) }$.

\begin{lemma}
\label{lem2.3}For any smooth symmetric $\left( 0,2\right) $ tensor $h$, let $%
I$ and $II$ be defined in (\ref{eq2.7}), then%
\begin{equation}
I\left( p,q,h\right) =-\int_{M}P_{g,h}^{\left( 1\right) }G_{q}\cdot
G_{p}d\mu -\frac{1}{2}G\left( p,q\right) trh\left( q\right) ,  \label{eq2.10}
\end{equation}%
and%
\begin{eqnarray}
&&II\left( p,p,h\right)  \label{eq2.11} \\
&=&-\int_{M}\left( P_{g,h}^{\left( 2\right) }G_{p}\cdot
G_{p}+P_{g,h}^{\left( 1\right) }G_{p}\cdot I_{p}+\frac{1}{2}P_{g,h}^{\left(
1\right) }G_{p}\cdot G_{p}trh\right) d\mu  \notag \\
&&-\frac{1}{2}I\left( p,p,h\right) trh\left( p\right) -\frac{1}{8}G\left(
p,p\right) \left( trh\left( p\right) \right) ^{2}+\frac{1}{4}G\left(
p,p\right) \left\vert h\left( p\right) \right\vert ^{2}.  \notag
\end{eqnarray}%
Here $G_{p}\left( q\right) =G\left( p,q\right) $, $I_{p}\left( q,h\right)
=I\left( p,q,h\right) $. The integration should be understood in
distribution sense.
\end{lemma}

\begin{proof}
For any smooth function $\varphi $ we have%
\begin{equation*}
\varphi \left( p\right) =\int_{M}P_{g+th}\varphi \cdot G_{g+th,p}d\mu
_{g+th},
\end{equation*}%
expand everything into power series of $t$, using%
\begin{equation*}
d\mu _{g+th}=\left[ 1+\frac{trh}{2}\cdot t+\left( \frac{\left( trh\right)
^{2}}{8}-\frac{\left\vert h\right\vert ^{2}}{4}\right) t^{2}+O\left(
t^{3}\right) \right] d\mu ,
\end{equation*}%
we see%
\begin{equation*}
\int_{M}\left( I_{p}\cdot P\varphi +G_{p}P_{g,h}^{\left( 1\right) }\varphi +%
\frac{1}{2}G_{p}trh\cdot P\varphi \right) d\mu =0
\end{equation*}%
and%
\begin{eqnarray*}
0 &=&\int_{M}\left[ P\varphi \cdot II_{p}+P_{g,h}^{\left( 1\right) }\varphi
\cdot I_{p}+P_{g,h}^{\left( 2\right) }\varphi \cdot G_{p}+\frac{1}{2}%
P\varphi \cdot I_{p}trh+\frac{1}{2}P_{g,h}^{\left( 1\right) }\varphi \cdot
G_{p}trh\right. \\
&&\left. +\frac{1}{8}P\varphi \cdot G_{p}\left( trh\right) ^{2}-\frac{1}{4}%
P\varphi \cdot G_{p}\left\vert h\right\vert ^{2}\right] d\mu .
\end{eqnarray*}%
By approximation we know the same formula remains true for $\varphi \in
H^{2}\left( M\right) $. Let $\varphi =G_{q}$ or $G_{p}$, we get the lemma.
\end{proof}

\section{First and second variation of Green's function pole's value\label%
{sec3}}

Let $N$ be the north pole on $S^{3}$ and $\pi _{N}:S^{3}\backslash \left\{
N\right\} \rightarrow \mathbb{R}^{3}$ be the stereographic projection. Using 
$x=\pi _{N}$ as the coordinate, we have the standard metric $g$ on $S^{3}$
can be written as%
\begin{equation}
g=\frac{4}{\left( \left\vert x\right\vert ^{2}+1\right) ^{2}}\left\vert
dx\right\vert ^{2}=\tau ^{-4}\left\vert dx\right\vert ^{2},  \label{eq3.1}
\end{equation}%
here%
\begin{equation}
\tau =\sqrt{\frac{\left\vert x\right\vert ^{2}+1}{2}}.  \label{eq3.2}
\end{equation}%
By conformal invariance property (\ref{eq1.3}), the Green's function of $P$
with pole at $N$ is given by%
\begin{equation}
G_{N}=-\frac{1}{4\pi }\frac{1}{\sqrt{\left\vert x\right\vert ^{2}+1}}.
\label{eq3.3}
\end{equation}%
More generally%
\begin{equation}
G\left( x,y\right) =-\frac{1}{4\pi }\frac{\left\vert x-y\right\vert }{\sqrt{%
\left\vert x\right\vert ^{2}+1}\sqrt{\left\vert y\right\vert ^{2}+1}}.
\label{eq3.4}
\end{equation}%
We are ready to compute the first variation of Green's function pole value.

\begin{proposition}
\label{prop3.1}For any $p\in S^{3}$ and smooth symmetric $\left( 0,2\right) $
tensor $h$, $I\left( p,p,h\right) =0$.
\end{proposition}

Proposition \ref{prop1.1} follows from Proposition \ref{prop3.1} and (\ref%
{eq2.8}).

\begin{proof}
By symmetry we can assume $p=N$. For convenience we write%
\begin{equation*}
I\left( h\right) =I\left( N,N,h\right) .
\end{equation*}%
Because we need to discuss various function's behavior near $N$, we denote $%
S $ as the south pole of $S^{3}$, $\pi _{S}:S^{3}\backslash \left\{
S\right\} \rightarrow \mathbb{R}^{3}$ as the stereographic projection with
respect to $S$. We can use $y=\pi _{S}$ as the coordinate. By Lemma \ref%
{lem2.3} and the fact $G_{N}\left( N\right) =0$,%
\begin{equation*}
I\left( h\right) =-\int_{S^{3}}P_{g,h}^{\left( 1\right) }G_{N}\cdot
G_{N}d\mu \text{ (in distribution sense).}
\end{equation*}%
Let $\eta \in C^{\infty }\left( \mathbb{R}^{3}\right) $ such that $\left.
\eta \right\vert _{B_{1}}=1$, $\left. \eta \right\vert _{\mathbb{R}%
^{3}\backslash B_{2}}=0$ and $0\leq \eta \leq 1$. For $\varepsilon >0$, we
write $\eta _{\varepsilon }=\eta \left( \frac{y}{\varepsilon }\right) $. By 
\cite[Lemma 2.2]{HY1} $\eta _{\varepsilon }G_{N}\rightarrow 0$ in $%
H^{2}\left( S^{3}\right) $, hence%
\begin{eqnarray*}
I\left( h\right) &=&-\lim_{\varepsilon \rightarrow
0^{+}}\int_{S^{3}}P_{g,h}^{\left( 1\right) }\left( \left( 1-\eta
_{\varepsilon }\right) G_{N}\right) \cdot G_{N}d\mu \\
&=&-\int_{S^{3}}P_{g,h}^{\left( 1\right) }G_{N}\cdot G_{N}d\mu \text{ (in
pointwise product sense).}
\end{eqnarray*}%
Note here we have used the dominated convergence theorem and the fact near $%
N $,%
\begin{equation*}
\left\vert P_{g,h}^{\left( 1\right) }G_{N}\cdot G_{N}\right\vert \leq
c\left\vert y\right\vert ^{-2},\quad \left\vert P_{g,h}^{\left( 1\right)
}\left( \left( 1-\eta _{\varepsilon }\right) G_{N}\right) \cdot
G_{N}\right\vert \leq c\left\vert y\right\vert ^{-2}
\end{equation*}%
here $c$ is independent of $\varepsilon $. For convenience we denote $\theta
=\tau ^{4}h$. By Lemma \ref{lem2.1} we have%
\begin{eqnarray*}
&&I\left( h\right) \\
&=&-\lim_{\varepsilon \rightarrow 0^{+}}\int_{S^{3}\backslash B_{\varepsilon
}\left( N\right) }P_{g,h}^{\left( 1\right) }G_{N}\cdot G_{N}d\mu \\
&=&-\frac{1}{32\pi ^{2}}\lim_{R\rightarrow \infty }\int_{\left\vert
x\right\vert \leq R}P_{\left\vert dx\right\vert ^{2},\theta }^{\left(
1\right) }1dx \\
&=&-\frac{1}{256\pi ^{2}}\lim_{R\rightarrow \infty }\int_{\left\vert
x\right\vert \leq R}\Delta \left( \theta _{ijij}-\Delta tr\theta \right) dx
\\
&=&-\frac{1}{256\pi ^{2}}\lim_{R\rightarrow \infty }\int_{\left\vert
x\right\vert =R}\left( \theta _{ijij}-\Delta tr\theta \right) _{k}\frac{x_{k}%
}{R}dS
\end{eqnarray*}%
To understand the boundary term we use the following notation: let $f$ be a
smooth function defined outside a ball, we say $f=O^{\left( \infty \right)
}\left( \left\vert x\right\vert ^{a}\right) $ as $\left\vert x\right\vert
\rightarrow \infty $ if for any $m$, $\partial _{i_{1}\cdots i_{m}}f\left(
x\right) =O\left( \left\vert x\right\vert ^{a-m}\right) $. In particular $%
h_{ij}=O^{\left( \infty \right) }\left( \left\vert x\right\vert ^{-4}\right) 
$, $\tau =O^{\left( \infty \right) }\left( \left\vert x\right\vert \right) $%
, hence $\theta _{ij}=O^{\left( \infty \right) }\left( 1\right) $ and%
\begin{equation*}
\left( \theta _{ijij}-\Delta tr\theta \right) _{k}=O^{\left( \infty \right)
}\left( \left\vert x\right\vert ^{-3}\right) ,
\end{equation*}%
this implies $I\left( h\right) =0$.
\end{proof}

It is worth pointing out that there are other ways to calculate $I\left(
N,N,h\right) $. For example one may do this by using the formula of $%
P_{g,h}^{\left( 1\right) }$ on $S^{3}$ (see (\ref{eqa.16})). However the
method in the above proof will be crucial for the calculation of second
variation formula.

To continue we need the expression of $I\left( N,q,h\right) $.

\begin{lemma}
\label{lem3.1}Let $\theta =\tau ^{4}h$, under the stereographic projection
with respect to $N$, we denote the coordinate of $q$ as $y$, then%
\begin{eqnarray}
&&I\left( N,q,h\right)  \label{eq3.5} \\
&=&-\frac{1}{256\pi ^{2}}\int_{\mathbb{R}^{3}}\left( \theta _{ijij}-\Delta
tr\theta \right) \cdot \frac{1}{\sqrt{\left\vert y\right\vert ^{2}+1}}\left( 
\frac{2}{\left\vert x-y\right\vert }-\frac{2}{\sqrt{\left\vert x\right\vert
^{2}+1}}-\frac{1}{\left( \sqrt{\left\vert x\right\vert ^{2}+1}\right) ^{3}}%
\right) dx  \notag \\
&&-\frac{G\left( N,q\right) }{8}\int_{S^{3}}\left( \Delta G_{N}-\frac{5}{2}%
G_{N}\right) h_{ijij}d\mu -\frac{G\left( N,q\right) }{8}\int_{S^{3}}\left(
\Delta G_{N}+\frac{5}{16}G_{N}\right) trhd\mu  \notag \\
&&+\frac{1}{8}G\left( N,q\right) trh\left( N\right) .  \notag
\end{eqnarray}
\end{lemma}

\begin{proof}
Indeed it follows from Lemma \ref{lem2.3} that%
\begin{eqnarray*}
&&I\left( N,q,h\right) \\
&=&I\left( q,N,h\right) \\
&=&-\int_{S^{3}}P_{g,h}^{\left( 1\right) }G_{N}\cdot G_{q}d\mu -\frac{1}{2}%
G\left( N,q\right) trh\left( N\right) \\
&=&-\int_{S^{3}}P_{g,h}^{\left( 1\right) }G_{N}\cdot \left(
G_{q}-G_{q}\left( N\right) \right) d\mu -G\left( N,q\right)
\int_{S^{3}}P_{g,h}^{\left( 1\right) }G_{N}d\mu -\frac{1}{2}G\left(
N,q\right) trh\left( N\right) .
\end{eqnarray*}%
By Lemma \ref{lem2.2} we have%
\begin{eqnarray*}
&&\int_{S^{3}}P_{g,h}^{\left( 1\right) }G_{N}d\mu \\
&=&\int_{S^{3}}G_{N}P_{g,h}^{\left( 1\right) }1d\mu -\frac{1}{2}%
\int_{S^{3}}\left( PG_{N}-G_{N}\cdot P1\right) trhd\mu \\
&=&\int_{S^{3}}G_{N}\left( \frac{1}{8}\Delta \left( h_{ijij}\right) -\frac{5%
}{16}h_{ijij}\right) d\mu +\frac{1}{8}\int_{S^{3}}\left( \Delta G_{N}+\frac{5%
}{16}G_{N}\right) trhd\mu -\frac{5}{8}trh\left( N\right) .
\end{eqnarray*}%
Using the fact $G_{q}-G_{q}\left( N\right) $ vanishes at $N$, by the same
method in the proof of Proposition \ref{prop3.1},%
\begin{eqnarray*}
&&\int_{S^{3}}P_{g,h}^{\left( 1\right) }G_{N}\cdot \left( G_{q}-G_{q}\left(
N\right) \right) d\mu \\
&=&\lim_{\varepsilon \rightarrow 0^{+}}\int_{S^{3}\backslash B_{\varepsilon
}\left( N\right) }P_{g,h}^{\left( 1\right) }G_{N}\cdot \left(
G_{q}-G_{q}\left( N\right) \right) d\mu \\
&=&\frac{1}{32\pi ^{2}}\lim_{R\rightarrow \infty }\int_{\left\vert
x\right\vert \leq R}P_{\left\vert dx\right\vert ^{2},\theta }1\cdot \frac{%
\left\vert x-y\right\vert -\sqrt{\left\vert x\right\vert ^{2}+1}}{\sqrt{%
\left\vert y\right\vert ^{2}+1}}dx \\
&=&\frac{1}{256\pi ^{2}}\int_{\mathbb{R}^{3}}\Delta \left( \theta
_{ijij}-\Delta tr\theta \right) \cdot \frac{\left\vert x-y\right\vert -\sqrt{%
\left\vert x\right\vert ^{2}+1}}{\sqrt{\left\vert y\right\vert ^{2}+1}}dx \\
&=&\frac{1}{256\pi ^{2}}\int_{\mathbb{R}^{3}}\left( \theta _{ijij}-\Delta
tr\theta \right) \cdot \frac{1}{\sqrt{\left\vert y\right\vert ^{2}+1}}\left( 
\frac{2}{\left\vert x-y\right\vert }-\frac{2}{\sqrt{\left\vert x\right\vert
^{2}+1}}-\frac{1}{\left( \sqrt{\left\vert x\right\vert ^{2}+1}\right) ^{3}}%
\right) dx.
\end{eqnarray*}%
Equation (\ref{eq3.5}) follows.
\end{proof}

\begin{theorem}
\label{thm3.1}For any smooth symmetric $\left( 0,2\right) $ tensor $h$,
denote $\theta =\tau ^{4}h$, then%
\begin{eqnarray}
&&II\left( N,N,h\right)  \label{eq3.6} \\
&=&-\frac{1}{128\pi ^{2}}\int_{\mathbb{R}^{3}}\left( \sum_{ij}\left( \theta
_{ikjk}+\theta _{jkik}-\left( tr\theta \right) _{ij}-\Delta \theta
_{ij}\right) ^{2}-\frac{3}{2}\left( \theta _{ijij}-\Delta tr\theta \right)
^{2}\right) dx.  \notag
\end{eqnarray}
\end{theorem}

Theorem \ref{thm1.1} follows from Theorem \ref{thm3.1} and (\ref{eq2.9}).

\begin{proof}
By Lemma \ref{lem2.3},%
\begin{eqnarray}
&&II\left( N,N,h\right)  \label{eq3.7} \\
&=&-\int_{S^{3}}\left( P_{g,h}^{\left( 2\right) }G_{N}\cdot
G_{N}+P_{g,h}^{\left( 1\right) }G_{N}\cdot I_{N}+\frac{1}{2}P_{g,h}^{\left(
1\right) }G_{N}\cdot G_{N}trh\right) d\mu .  \notag
\end{eqnarray}%
First we note that because $G_{N}\left( N\right) =0$, the same argument as
in the proof of Proposition \ref{prop3.1} shows%
\begin{eqnarray*}
&&\int_{S^{3}}P_{g,h}^{\left( 2\right) }G_{N}\cdot G_{N}d\mu \\
&=&\lim_{\varepsilon \rightarrow 0^{+}}\int_{S^{3}\backslash B_{\varepsilon
}\left( N\right) }P_{g,h}^{\left( 2\right) }G_{N}\cdot G_{N}d\mu \\
&=&\frac{1}{32\pi ^{2}}\lim_{R\rightarrow \infty }\int_{\left\vert
x\right\vert \leq R}P_{\left\vert dx\right\vert ^{2},\theta }^{\left(
2\right) }1dx \\
&=&\frac{1}{128\pi ^{2}}\int_{\mathbb{R}^{3}}\left[ \sum_{ij}\left( \theta
_{ikjk}+\theta _{jkik}-\left( tr\theta \right) _{ij}-\Delta \theta
_{ij}\right) ^{2}\right. \\
&&\left. -\frac{1}{16}\left( 23\theta _{ijij}-19\Delta tr\theta \right)
\left( \theta _{klkl}-\Delta tr\theta \right) \right] dx,
\end{eqnarray*}%
Here we have used (\ref{eq2.5}). Next using $I_{N}\left( N\right) =0$ we have%
\begin{eqnarray*}
\int_{S^{3}}P_{g,h}^{\left( 1\right) }G_{N}\cdot I_{N}d\mu
&=&\lim_{\varepsilon \rightarrow 0^{+}}\int_{S^{3}\backslash B_{\varepsilon
}\left( N\right) }P_{g,h}^{\left( 1\right) }G_{N}\cdot I_{N}d\mu \\
&=&-\frac{1}{4\pi \sqrt{2}}\lim_{R\rightarrow \infty }\int_{\left\vert
x\right\vert \leq R}P_{\left\vert dx\right\vert ^{2},\theta }^{\left(
1\right) }1\cdot \tau I_{N}dx.
\end{eqnarray*}%
Since%
\begin{equation*}
\lim_{R\rightarrow \infty }\int_{\left\vert x\right\vert \leq
R}P_{\left\vert dx\right\vert ^{2},\theta }^{\left( 1\right) }1\cdot \tau
G_{N}dx=0,
\end{equation*}%
(see the proof of Proposition \ref{prop3.1}), by Lemma \ref{lem3.1} we have%
\begin{eqnarray*}
&&\int_{S^{3}}P_{g,h}^{\left( 1\right) }G_{N}\cdot I_{N}d\mu \\
&=&\frac{1}{2048\pi ^{3}}\int_{\mathbb{R}^{3}}P_{\left\vert dx\right\vert
^{2},\theta }^{\left( 1\right) }1\left( \int_{\mathbb{R}^{3}}\left( \theta
_{ijij}-\Delta tr\theta \right) \left( y\right) \left( \frac{2}{\left\vert
x-y\right\vert }-\frac{2}{\sqrt{\left\vert y\right\vert ^{2}+1}}-\frac{1}{%
\left( \sqrt{\left\vert y\right\vert ^{2}+1}\right) ^{3}}\right) dy\right) dx
\\
&=&\frac{1}{16384\pi ^{3}}\int_{\mathbb{R}^{3}}dy\left( \theta
_{ijij}-\Delta tr\theta \right) \left( y\right) \int_{\mathbb{R}^{3}}\Delta
\left( \theta _{klkl}-\Delta tr\theta \right) \left( x\right) \left( \frac{2%
}{\left\vert x-y\right\vert }-\frac{2}{\sqrt{\left\vert y\right\vert ^{2}+1}}%
-\frac{1}{\left( \sqrt{\left\vert y\right\vert ^{2}+1}\right) ^{3}}\right) dx
\\
&=&-\frac{1}{2048\pi ^{2}}\int_{\mathbb{R}^{3}}\left( \theta _{ijij}-\Delta
tr\theta \right) ^{2}dx.
\end{eqnarray*}%
Similarly for the third term in (\ref{eq3.7}) we have%
\begin{eqnarray*}
\int_{S^{3}}\frac{1}{2}P_{g,h}^{\left( 1\right) }G_{N}\cdot G_{N}trhd\mu &=&%
\frac{1}{2}\lim_{\varepsilon \rightarrow 0^{+}}\int_{S^{3}\backslash
B_{\varepsilon }\left( N\right) }P_{g,h}^{\left( 1\right) }G_{N}\cdot
G_{N}trhd\mu \\
&=&\frac{1}{64\pi ^{2}}\lim_{R\rightarrow \infty }\int_{\left\vert
x\right\vert \leq R}P_{\left\vert dx\right\vert ^{2},\theta }^{\left(
1\right) }1\cdot tr\theta dx \\
&=&\frac{1}{512\pi ^{2}}\int_{\mathbb{R}^{3}}\Delta \left( \theta
_{ijij}-\Delta tr\theta \right) \cdot tr\theta dx \\
&=&\frac{1}{512\pi ^{2}}\int_{\mathbb{R}^{3}}\Delta tr\theta \cdot \left(
\theta _{ijij}-\Delta tr\theta \right) dx.
\end{eqnarray*}%
Sum up we get (\ref{eq3.6}).
\end{proof}

Next we will study sign of the second variation. For convenience we write%
\begin{equation}
II\left( h\right) =II\left( N,N,h\right) .  \label{eq3.8}
\end{equation}%
First we observe that by conformal covariant property, for any smooth vector
field $X$ and function $f$,%
\begin{equation}
II\left( L_{X}g+fg\right) =0.  \label{eq3.9}
\end{equation}%
Indeed let $\phi _{t}$ be the flow generated by $X$, then for $t$ near $0$,%
\begin{eqnarray*}
G_{\left( 1+tf\right) \phi _{t}^{\ast }g}\left( N,N\right) &=&\left(
1+tf\left( N\right) \right) ^{\frac{1}{2}}G_{\phi _{t}^{\ast }g}\left(
N,N\right) \\
&=&\left( 1+tf\left( N\right) \right) ^{\frac{1}{2}}G_{g}\left( \phi
_{t}\left( N\right) ,\phi _{t}\left( N\right) \right) =0.
\end{eqnarray*}%
Since $\left. \partial _{t}\right\vert _{t=0}\left( \left( 1+tf\right) \phi
_{t}^{\ast }g\right) =L_{X}g+fg$, (\ref{eq3.9}) follows. In fact we can say
a little more: let $II\left( h,k\right) $ be the symmetric form associated
with $II\left( h\right) $, if%
\begin{equation}
\theta =\tau ^{4}h,\quad \kappa =\tau ^{4}k,  \label{eq3.10}
\end{equation}%
then it follows from Theorem \ref{thm3.1} that%
\begin{eqnarray}
&&II\left( h,k\right)  \label{eq3.11} \\
&=&-\frac{1}{128\pi ^{2}}\int_{\mathbb{R}^{3}}\left[ \sum_{ij}\left( \theta
_{ikjk}+\theta _{jkik}-\left( tr\theta \right) _{ij}-\Delta \theta
_{ij}\right) \left( \kappa _{iljl}+\kappa _{jlil}-\left( tr\kappa \right)
_{ij}-\Delta \kappa _{ij}\right) \right.  \notag \\
&&\left. -\frac{3}{2}\left( \theta _{ijij}-\Delta tr\theta \right) \left(
\kappa _{klkl}-\Delta tr\kappa \right) \right] dx.  \notag
\end{eqnarray}

\begin{lemma}
\label{lem3.2}Given smooth symmetric $\left( 0,2\right) $ tensor $h$, vector
field $X$ and function $f$, we have%
\begin{equation}
II\left( h,L_{X}g+fg\right) =0.  \label{eq3.12}
\end{equation}
\end{lemma}

To achieve this we need the following technical fact:

\begin{lemma}
\label{lem3.3}If $h$ is a smooth symmetric $\left( 0,2\right) $ tensor on $%
S^{3}$, then there exists a smooth vector field $X$ such that%
\begin{equation*}
\left( h-L_{X}g\right) \left( N\right) =0,\quad D\left( h-L_{X}g\right)
\left( N\right) =0
\end{equation*}
\end{lemma}

To derive this lemma, we start with the following linear algebra fact.

\begin{lemma}
\label{lem3.4}Denote%
\begin{equation*}
\mathcal{P}_{m}=\left\{ \text{homogeneous degree }m\text{ polynomials on }%
\mathbb{R}^{3}\right\} .
\end{equation*}%
If for $1\leq i,j\leq 3$, $H_{ij}\in \mathcal{P}_{1}$, $H_{ij}=H_{ji}$, then
there exists unique $A_{i}\in \mathcal{P}_{2}$ such that%
\begin{equation*}
\partial _{i}A_{j}+\partial _{j}A_{i}=H_{ij}.
\end{equation*}
\end{lemma}

\begin{proof}
Let%
\begin{equation*}
\mathcal{X}=\left\{ \left[ 
\begin{array}{c}
A_{1} \\ 
A_{2} \\ 
A_{3}%
\end{array}%
\right] :A_{i}\in \mathcal{P}_{2}\right\}
\end{equation*}%
and%
\begin{equation*}
\mathcal{Y}=\left\{ \left[ H_{ij}\right] _{1\leq i,j\leq 3}:H_{ij}\in 
\mathcal{P}_{1},H_{ij}=H_{ji}\right\} .
\end{equation*}%
Note $\dim \mathcal{X}=\dim \mathcal{Y}=18$. Let%
\begin{equation*}
\phi :\mathcal{X}\rightarrow \mathcal{Y}:\left[ 
\begin{array}{c}
A_{1} \\ 
A_{2} \\ 
A_{3}%
\end{array}%
\right] \mapsto \left[ H_{ij}\right] _{1\leq i,j\leq 3}
\end{equation*}%
be given by $H_{ij}=\partial _{i}A_{j}+\partial _{j}A_{i}$. We need to show $%
\phi $ is a linear isomorphism. We only need to prove $\ker \phi =0$. Indeed
if $\left[ 
\begin{array}{c}
A_{1} \\ 
A_{2} \\ 
A_{3}%
\end{array}%
\right] \in \ker \phi $, then%
\begin{equation*}
\partial _{i}A_{j}+\partial _{j}A_{i}=0.
\end{equation*}%
This implies%
\begin{equation*}
x_{i}x_{j}\partial _{i}A_{j}+x_{i}x_{j}\partial _{j}A_{i}=0.
\end{equation*}%
Since $x_{i}\partial _{i}A_{j}=2A_{j}$, we see $x_{i}A_{i}=0$. Hence%
\begin{equation*}
0=\partial _{j}\left( x_{i}A_{i}\right) =A_{j}+x_{i}\partial
_{j}A_{i}=A_{j}-x_{i}\partial _{i}A_{j}=-A_{j}.
\end{equation*}%
The lemma follows.
\end{proof}

Now we will use Taylor expansion to prove Lemma \ref{lem3.3}.

\begin{proof}[Proof of Lemma \protect\ref{lem3.3}]
By standard cut-off argument we see the conclusion is in fact a local
statement. We choose a local coordinate near $N$, say $y_{1},y_{2},y_{3}$
such that $y_{i}\left( N\right) =0$. Assume $X=X^{i}\frac{\partial }{%
\partial y_{i}}$, let $\alpha $ be the associated $1$-form i.e. $\alpha
_{i}=g_{ij}X^{j}$, then we only need to find $\alpha $ with%
\begin{equation*}
\alpha _{ij}+\alpha _{ji}=h_{ij}+O\left( \left\vert y\right\vert ^{2}\right)
\end{equation*}%
as $y\rightarrow 0$. In another way the equation is%
\begin{equation*}
\partial _{i}\alpha _{j}+\partial _{j}\alpha _{i}-2\Gamma _{ij}^{k}\alpha
_{k}=h_{ij}+O\left( \left\vert y\right\vert ^{2}\right) .
\end{equation*}%
We will look for $\alpha _{i}=\alpha _{i}^{\left( 1\right) }+\alpha
_{i}^{\left( 2\right) }$, here $\alpha _{i}^{\left( l\right) }\in \mathcal{P}%
_{l}$. We have the Taylor expansion of $h_{ij}$ as $h_{ij}=h_{ij}\left(
0\right) +h_{ij}^{\left( 1\right) }+O\left( \left\vert y\right\vert
^{2}\right) $. So the equation becomes%
\begin{equation}
\partial _{i}\alpha _{j}^{\left( 1\right) }+\partial _{j}\alpha _{i}^{\left(
1\right) }=h_{ij}\left( 0\right)  \label{eq3.13}
\end{equation}%
and%
\begin{equation}
\partial _{i}\alpha _{j}^{\left( 2\right) }+\partial _{j}\alpha _{i}^{\left(
2\right) }-2\Gamma _{ij}^{k}\left( 0\right) \alpha _{k}^{\left( 1\right)
}=h_{ij}^{\left( 1\right) }.  \label{eq3.14}
\end{equation}%
For (\ref{eq3.13}), we can simply choose $\alpha _{i}^{\left( 1\right) }=%
\frac{1}{2}h_{ik}\left( 0\right) y_{k}$. Using Lemma \ref{lem3.4} we see (%
\ref{eq3.14}) also has a solution. Lemma \ref{lem3.3} follows.
\end{proof}

With Lemma \ref{lem3.3} at hand, we can proceed to prove Lemma \ref{lem3.2}.

\begin{proof}[Proof of Lemma \protect\ref{lem3.2}]
Note that%
\begin{equation*}
f=O^{\left( \infty \right) }\left( 1\right) ,\quad \theta _{ij}=O^{\left(
\infty \right) }\left( 1\right) ,
\end{equation*}%
as $\left\vert x\right\vert \rightarrow \infty $. By (\ref{eq3.10}),%
\begin{eqnarray*}
&&II\left( h,fg\right) \\
&=&-\frac{1}{128\pi ^{2}}\int_{\mathbb{R}^{3}}\left[ -\left( \theta
_{ikjk}+\theta _{jkik}-\left( tr\theta \right) _{ij}-\Delta \theta
_{ij}\right) \left( f_{ij}+\Delta f\cdot \delta _{ij}\right) \right. \\
&&\left. +3\left( \theta _{ijij}-\Delta tr\theta \right) \Delta f\right] dx.
\end{eqnarray*}%
By integration by parts, we have%
\begin{equation*}
\int_{\mathbb{R}^{3}}-\left( \theta _{ikjk}+\theta _{jkik}-\left( tr\theta
\right) _{ij}-\Delta \theta _{ij}\right) f_{ij}dx=\int_{\mathbb{R}%
^{3}}\left( -\theta _{ijij}+\Delta tr\theta \right) \Delta fdx
\end{equation*}%
and%
\begin{equation*}
\int_{\mathbb{R}^{3}}-\left( \theta _{ikjk}+\theta _{jkik}-\left( tr\theta
\right) _{ij}-\Delta \theta _{ij}\right) \Delta f\cdot \delta _{ij}dx=\int_{%
\mathbb{R}^{3}}\left( -2\theta _{ijij}+2\Delta tr\theta \right) \Delta fdx.
\end{equation*}%
Sum up we get $II\left( h,fg\right) =0$.

Next we will show $II\left( h,L_{X}g\right) =0$. First we note that it
follows from (\ref{eq3.9}) that for any smooth vector fields $X^{\prime },X$
and smooth functions $f^{\prime },f$, we have%
\begin{equation*}
II\left( L_{X^{\prime }}g+f^{\prime }g,L_{X}g+fg\right) =0.
\end{equation*}

To continue we can assume $h$ satisfies%
\begin{equation*}
h\left( N\right) =0,\quad Dh\left( N\right) =0.
\end{equation*}%
Indeed given any smooth $h$, by Lemma \ref{lem3.3}, we can find a smooth
vector field $X^{\prime }$ such that%
\begin{equation*}
\left( h-L_{X^{\prime }}g\right) \left( N\right) =0,\quad D\left(
h-L_{X^{\prime }}g\right) \left( N\right) =0.
\end{equation*}%
It follows that%
\begin{equation*}
II\left( h,L_{X}g\right) =II\left( h-L_{X^{\prime }}g,L_{X}g\right) .
\end{equation*}

Under the additional assumption on $h$, we have%
\begin{equation*}
\theta _{ij}=O^{\left( \infty \right) }\left( \left\vert x\right\vert
^{-2}\right) ,\quad X_{i}=O^{\left( \infty \right) }\left( \left\vert
x\right\vert ^{2}\right)
\end{equation*}%
as $x\rightarrow \infty $. Here $X=X_{i}\frac{\partial }{\partial x_{i}}$.
Let $\kappa =\tau ^{4}L_{X}g$, then%
\begin{equation*}
\kappa _{ij}=-4\tau ^{-1}X\tau \cdot \delta _{ij}+X_{ij}+X_{ji}.
\end{equation*}%
Hence%
\begin{equation*}
\kappa _{ikjk}+\kappa _{jkik}-\left( tr\kappa \right) _{ij}-\Delta \kappa
_{ij}=-4\left( \tau ^{-1}X\tau \right) _{ij}-4\Delta \left( \tau ^{-1}X\tau
\right) \cdot \delta _{ij}
\end{equation*}%
and%
\begin{equation*}
\kappa _{ijij}-\Delta tr\kappa =8\Delta \left( \tau ^{-1}X\tau \right) .
\end{equation*}%
It implies%
\begin{eqnarray*}
&&II\left( h,L_{X}g\right) \\
&=&\frac{1}{32\pi ^{2}}\int_{\mathbb{R}^{3}}\left[ -\left( \theta
_{ikjk}+\theta _{jkik}-\left( tr\theta \right) _{ij}-\Delta \theta
_{ij}\right) \left( \left( \tau ^{-1}X\tau \right) _{ij}+\Delta \left( \tau
^{-1}X\tau \right) \cdot \delta _{ij}\right) \right. \\
&&\left. +3\left( \theta _{ijij}-\Delta tr\theta \right) \Delta \left( \tau
^{-1}X\tau \right) \right] dx.
\end{eqnarray*}%
Note that $\tau ^{-1}X\tau =O^{\left( \infty \right) }\left( \left\vert
x\right\vert \right) $ and $\theta _{ij}=O^{\left( \infty \right) }\left(
\left\vert x\right\vert ^{-2}\right) $, the same integration by parts
argument as the beginning shows $II\left( h,L_{X}g\right) =0$.
\end{proof}

\begin{proposition}
\label{prop3.2}For any smooth symmetric $\left( 0,2\right) $ tensor $h$, $%
II\left( h\right) \leq 0$. Moreover, $II\left( h\right) =0$ if and only if $%
h=L_{X}g+f\cdot g$ for some smooth vector fields $X$ and smooth functions $f$
on $S^{3}$.
\end{proposition}

\begin{proof}
In view of Lemma \ref{lem3.2} and \ref{lem3.3} we can assume $h\left(
N\right) =0$ and $Dh\left( N\right) =0$. Under such assumption we have%
\begin{equation*}
\theta _{ij}=O^{\left( \infty \right) }\left( \left\vert x\right\vert
^{-2}\right)
\end{equation*}%
as $x\rightarrow \infty $. In particular $\theta _{ij}\in L^{2}\left( 
\mathbb{R}^{3}\right) $. Let%
\begin{equation*}
a_{ij}\left( \xi \right) =\widehat{\theta _{ij}}\left( \xi \right) ,
\end{equation*}%
then $a_{ij}\in L^{2}\left( \mathbb{R}^{3}\right) $. By Parseval relation we
have%
\begin{eqnarray*}
&&\int_{\mathbb{R}^{3}}\left( \sum_{ij}\left( \theta _{ikjk}+\theta
_{jkik}-\left( tr\theta \right) _{ij}-\Delta \theta _{ij}\right) ^{2}-\frac{3%
}{2}\left( \theta _{ijij}-\Delta tr\theta \right) ^{2}\right) dx \\
&=&\int_{\mathbb{R}^{3}}\left( -2a_{ij}\overline{a}_{ik}\xi _{j}\xi
_{k}\left\vert \xi \right\vert ^{2}+\frac{1}{2}a_{ij}\overline{a}_{kl}\xi
_{i}\xi _{j}\xi _{k}\xi _{l}+\frac{1}{2}a_{ij}\xi _{i}\xi _{j}\overline{a}%
_{kk}\left\vert \xi \right\vert ^{2}+\frac{1}{2}\overline{a}_{ij}\xi _{i}\xi
_{j}a_{kk}\left\vert \xi \right\vert ^{2}\right. \\
&&\left. +\sum_{ij}\left\vert a_{ij}\right\vert ^{2}\left\vert \xi
\right\vert ^{4}-\frac{1}{2}\left\vert \sum_{i}a_{ii}\right\vert
^{2}\left\vert \xi \right\vert ^{4}\right) d\xi .
\end{eqnarray*}%
We will show the integrand in nonnegative. Indeed, if we write%
\begin{equation*}
A=\left[ a_{ij}\right] _{1\leq i,j\leq 3},
\end{equation*}%
then $A$ is symmetric and the integrand is equal to%
\begin{eqnarray}
&&-2\left\vert A\xi \right\vert ^{2}\left\vert \xi \right\vert ^{2}+\frac{1}{%
2}\left\vert \xi ^{T}A\xi \right\vert ^{2}+\frac{1}{2}\xi ^{T}A\xi \cdot 
\overline{trA}\left\vert \xi \right\vert ^{2}+\frac{1}{2}\xi ^{T}\overline{A}%
\xi \cdot trA\left\vert \xi \right\vert ^{2}  \label{eq3.15} \\
&&-\frac{1}{2}\left\vert trA\right\vert ^{2}\left\vert \xi \right\vert
^{4}+\left\vert A\right\vert ^{2}\left\vert \xi \right\vert ^{4}.  \notag
\end{eqnarray}%
Assume $\xi \neq 0$, then we may find an orthogonal matrix $O$ such that%
\begin{equation*}
\xi =O\left[ 
\begin{array}{c}
\left\vert \xi \right\vert \\ 
0 \\ 
0%
\end{array}%
\right] .
\end{equation*}%
Denote%
\begin{equation*}
B=O^{T}AO,
\end{equation*}%
then $B$ is symmetric and the integrand (\ref{eq3.15}) is equal to%
\begin{equation}
\left( \frac{1}{2}\left\vert b_{22}-b_{33}\right\vert ^{2}+2\left\vert
b_{23}\right\vert ^{2}\right) \left\vert \xi \right\vert ^{4}  \label{eq3.16}
\end{equation}%
and hence nonnegative. This implies $II\left( h\right) \leq 0$.

If $II\left( h\right) =0$, then we have $b_{22}=b_{33}$ and $b_{23}=0$. this
implies%
\begin{equation}
a_{ij}=\alpha \delta _{ij}+\beta _{i}\xi _{j}+\beta _{j}\xi _{i},
\label{eq3.17}
\end{equation}%
here $\alpha $ and $\beta _{i}$ depend on $\xi $. To continue we recall the
orthogonal decomposition \cite[p130, lemma 4.57]{B},%
\begin{equation}
\mathcal{S}^{2}S^{3}=\mathcal{A}\oplus \mathcal{B}.  \label{eq3.18}
\end{equation}%
Here%
\begin{eqnarray*}
\mathcal{S}^{2}S^{3} &=&\left\{ C^{\infty }\text{ symmetric }\left(
0,2\right) \text{ tensors on }S^{3}\right\} , \\
\mathcal{A} &=&\left\{ L_{X}g+fg:X\text{ is a }C^{\infty }\text{ vector
field, }f\text{ is a }C^{\infty }\text{ function}\right\} , \\
\mathcal{B} &=&\left\{ k\in \mathcal{S}^{2}S^{3}:trk=0,k_{ijj}=0\right\} .
\end{eqnarray*}%
To show $h\in \mathcal{A}$, we only need to prove $h\perp \mathcal{B}$.
Indeed if $k\in \mathcal{B}$, let $\kappa =\tau ^{4}k=\kappa
_{ij}dx_{i}dx_{j}$, then%
\begin{equation}
\kappa _{ii}=0,\quad \left( \tau ^{-6}\kappa _{ij}\right) _{j}=0.
\label{eq3.19}
\end{equation}%
On the other hand we have%
\begin{equation}
\theta _{ij}=O^{\left( \infty \right) }\left( \left\vert x\right\vert
^{-2}\right) ,\quad \tau ^{-6}\kappa _{ij}=O^{\left( \infty \right) }\left(
\left\vert x\right\vert ^{-6}\right) .  \label{eq3.20}
\end{equation}%
By Fourier transform we have%
\begin{equation}
\widehat{\tau ^{-6}\kappa _{ii}}=0,\quad \widehat{\tau ^{-6}\kappa _{ij}}\xi
_{j}=0.  \label{eq3.21}
\end{equation}%
Hence%
\begin{eqnarray*}
\int_{S^{3}}\left\langle h,k\right\rangle d\mu &=&\int_{\mathbb{R}%
^{3}}\theta _{ij}\kappa _{ij}\tau ^{-6}dx \\
&=&\int_{\mathbb{R}^{3}}\widehat{\theta _{ij}}\overline{\widehat{\tau
^{-6}\kappa _{ij}}}d\xi \\
&=&\int_{\mathbb{R}^{3}}\left( \alpha \delta _{ij}+\beta _{i}\xi _{j}+\beta
_{j}\xi _{i}\right) \overline{\widehat{\tau ^{-6}\kappa _{ij}}}d\xi \\
&=&0,
\end{eqnarray*}%
here we have used (\ref{eq3.21}) in the last step. Hence $h=L_{X}g+fg$ for
some smooth vector field $X$ and smooth function $f$.
\end{proof}

\section{A new invariant for Paneitz operator\label{sec4}}

Let $\left( M,g\right) $ be a smooth compact three dimensional Riemannian
manifold. For any $p\in M$, we set%
\begin{equation}
\nu _{p}=\inf \left\{ \frac{E\left( u\right) }{\int_{M}u^{2}d\mu }:u\in
H^{2}\left( M\right) \backslash \left\{ 0\right\} ,u\left( p\right)
=0\right\} .  \label{eq4.1}
\end{equation}%
$\nu _{p}$ is always finite and achieved. Indeed we let $u_{i}\in
H^{2}\left( M\right) $ such that $u_{i}\left( p\right) =0$, $\left\Vert
u_{i}\right\Vert _{L^{2}}=1$ and $E\left( u_{i}\right) \rightarrow \nu _{p}$%
. In view of the fact%
\begin{equation*}
E\left( u_{i}\right) \geq c_{1}\left\Vert u_{i}\right\Vert _{H^{2}\left(
M\right) }^{2}-c_{2}\left\Vert u_{i}\right\Vert _{L^{2}}^{2}
\end{equation*}%
for some positive constants $c_{1}$ and $c_{2}$, we see $\left\Vert
u_{i}\right\Vert _{H^{2}\left( M\right) }^{2}\leq c$, independent of $i$.
After passing to a subsequence we can assume $u_{i}\rightharpoonup u$ weakly
in $H^{2}\left( M\right) $. It follows that $u_{i}\rightarrow u$ uniformly
and hence $u\left( p\right) =0$ and $\left\Vert u\right\Vert _{L^{2}}=1$. By
lower semicontinuity we have%
\begin{equation*}
E\left( u\right) \leq \lim \inf_{i\rightarrow \infty }E\left( u_{i}\right)
=\nu _{p}.
\end{equation*}%
Hence $E\left( u\right) =\nu _{p}$ and $u$ is a minimizer.

Note $u$ satisfies%
\begin{equation*}
E\left( u,\varphi \right) =\nu _{p}\int_{M}u\varphi d\mu
\end{equation*}%
for any $\varphi \in H^{2}\left( M\right) $ with $\varphi \left( p\right) =0$%
. For any $\psi \in H^{2}\left( M\right) $, let $\varphi =\psi -\psi \left(
p\right) $ we see%
\begin{equation*}
E\left( u,\psi \right) =\nu _{p}\int_{M}u\psi d\mu +\alpha \psi \left(
p\right) .
\end{equation*}%
Here $\alpha $ is a constant. In another word, we have%
\begin{equation}
Pu=\nu _{p}u+\alpha \delta _{p}  \label{eq4.2}
\end{equation}%
in distribution sense and%
\begin{equation}
u\in H^{2}\left( M\right) ,\left\Vert u\right\Vert _{L^{2}}=1,u\left(
p\right) =0.  \label{eq4.3}
\end{equation}%
Sometime to avoid confusion we write $u=u_{p}$ and $\alpha =\alpha _{p}$.

\begin{example}
\label{ex4.1}Using \cite[Lemma 7.1 and Corollary 7.1]{HY1}, we see on
standard $S^{3}$, $\nu _{N}=0$ and it is achieved on constant multiple of
the Green's function $G_{N}$. Calculation shows $\left\Vert G_{N}\right\Vert
_{L^{2}}=\frac{1}{4}$, hence $u_{N}=4G_{N}$, $\alpha _{N}=4$.
\end{example}

Let%
\begin{eqnarray}
\nu \left( M,g\right) &=&\inf_{p\in M}\nu \left( M,g,p\right)  \label{eq4.4}
\\
&=&\inf \left\{ \frac{E\left( u\right) }{\int_{M}u^{2}d\mu }:u\in
H^{2}\left( M\right) \backslash \left\{ 0\right\} ,u\text{ vanishes somewhere%
}\right\} .  \notag
\end{eqnarray}%
We will write $\nu \left( g\right) $ when no confusion could happen. Same
argument as before shows $\nu \left( g\right) $ is finite and achieved. It
is clear that condition P is satisfied if and only if $\nu \left( g\right)
>0 $, condition NN is satisfied if and only if $\nu \left( g\right) \geq 0$.
By Example \ref{ex4.1} and symmetry, we see $\nu \left(
S^{3},g_{S^{3}}\right) =0$.

Here we make some general discussion about $\nu _{p}$ and $\nu \left(
g\right) $. For convenience we write $\nu =\nu \left( g\right) $. Let%
\begin{equation*}
\lambda _{1}\leq \lambda _{2}\leq \lambda _{3}\leq \cdots
\end{equation*}%
be eigenvalues and $\varphi _{i}$ be the associated orthonormal
eigenfunctions of Paneitz operator $P$, then%
\begin{equation}
\lambda _{1}\leq \nu _{p}\leq \lambda _{2},\quad \lambda _{1}\leq \nu \leq
\lambda _{2}.  \label{eq4.5}
\end{equation}

Indeed, for given $p\in M$, there exists $c_{1},c_{2}$ not all zeroes such
that $c_{1}\varphi _{1}\left( p\right) +c_{2}\varphi _{2}\left( p\right) =0$%
, then%
\begin{equation*}
\nu _{p}\leq \frac{E\left( c_{1}\varphi _{1}+c_{2}\varphi _{2}\right) }{%
\left\Vert c_{1}\varphi _{1}+c_{2}\varphi _{2}\right\Vert _{L^{2}}^{2}}=%
\frac{\lambda _{1}c_{1}^{2}+\lambda _{2}c_{2}^{2}}{c_{1}^{2}+c_{2}^{2}}\leq
\lambda _{2}.
\end{equation*}%
The inequality of $\nu $ follows.

Assume $u\in H^{2}\left( M\right) $, $\left\Vert u\right\Vert _{L^{2}}=1$
and $u\left( p\right) =0$ for some $p$ with $E\left( u\right) =\nu $ i.e. $u$
is a minimizer for the $\nu $ problem, then%
\begin{equation}
Pu=\nu u+\alpha \delta _{p}.  \label{eq4.6}
\end{equation}%
If $\sharp u^{-1}\left( 0\right) >1$ i.e. $u$ vanishes at two or more
points, then%
\begin{equation}
Pu=\nu u,\quad \lambda _{1}=\nu .  \label{eq4.7}
\end{equation}%
Indeed assume $u\left( p_{1}\right) =0$ and $u\left( p_{2}\right) =0$ for $%
p_{1}\neq p_{2}$, then for $\varphi \in H^{2}$ with either $\varphi \left(
p_{1}\right) =0$ or $\varphi \left( p_{2}\right) =0$,%
\begin{equation}
E\left( u,\varphi \right) =\nu \int_{M}u\varphi d\mu .  \label{eq4.8}
\end{equation}%
Hence (\ref{eq4.8}) is valid for any $\varphi \in H^{2}$. In another word $%
Pu=\nu u$. If $\lambda _{1}<\nu $, then $\lambda _{1}<\lambda _{2}$ and $%
\varphi _{1}$ does not vanish anywhere. Using%
\begin{equation*}
\int_{M}\varphi _{1}\varphi _{2}d\mu =0,
\end{equation*}%
we see $\varphi _{2}$ must change sign. Hence for $\varepsilon >0$ small%
\begin{equation*}
\lambda _{1}<\nu \leq \frac{E\left( \varepsilon \varphi _{1}+\varphi
_{2}\right) }{\left\Vert \varepsilon \varphi _{1}+\varphi _{2}\right\Vert
_{L^{2}}^{2}}=\frac{\varepsilon ^{2}\lambda _{1}+\lambda _{2}}{\varepsilon
^{2}+1}<\lambda _{2}
\end{equation*}%
A contradiction with the fact $\nu $ is an eigenvalue. Hence $\nu $ must be
the first eigenvalue.

Now we can state the following interesting relation between condition NN and
the sign of $\lambda _{2}$.

\begin{proposition}
\label{prop4.1}Assume the Yamabe invariant $Y\left( g\right) >0$ and there
exists a $\widetilde{g}\in \left[ g\right] $ such that $\widetilde{Q}\geq 0$
and not identically zero, then the following statements are equivalent

\begin{enumerate}
\item $Y_{4}\left( g\right) >-\infty $.

\item $\lambda _{2}\left( P\right) >0$.

\item $\nu \left( g\right) \geq 0$ i.e. $P$ satisfies condition NN.
\end{enumerate}
\end{proposition}

\begin{proof}
It follows from the assumption that $\lambda _{1}<0$. By \cite[Proposition
1.2]{HY3} and (\ref{eq1.4}) we have $\ker P=0$ and $G_{P}\left( p,q\right)
<0 $ for $p\neq q$. Here $G_{P}$ is the Green's function of the Paneitz
operator. Let $m\geq 1$ be the natural number such that $\lambda _{m}<0$ and 
$\lambda _{m+1}>0$ i.e. $\lambda _{m}$ is the largest negative eigenvalue.
By applying the classical Krein-Rutman theorem to the operator%
\begin{equation*}
Tf\left( p\right) =-\int_{M}G_{P}\left( p,q\right) f\left( q\right) d\mu
\left( q\right)
\end{equation*}%
we know $\lambda _{m}$ must be simple and $\varphi _{m}$ can not touch zero
(see \cite[section 4]{HY3}). Without losing of generality, we assume $%
\varphi _{m}>0$.

\begin{itemize}
\item[(1)$\Rightarrow $(2):] If $\lambda _{2}<0$, then $m\geq 2$ and the
first eigenfunction $\varphi _{1}$ must change sign. Let%
\begin{equation*}
\kappa =-\min_{p\in M}\frac{\varphi _{1}\left( p\right) }{\varphi _{m}\left(
p\right) }>0,
\end{equation*}%
then $\varphi _{1}+\kappa \varphi _{m}\geq 0$ and it touches zero somewhere.
On the other hand $E\left( \varphi _{1}+\kappa \varphi _{m}\right) =\lambda
_{1}+\kappa ^{2}\lambda _{m}<0$, hence%
\begin{equation*}
Y_{4}\left( g\right) \leq \left\Vert \left( \varphi _{1}+\kappa \varphi
_{m}+\varepsilon \right) ^{-1}\right\Vert _{L^{6}}^{2}E\left( \varphi
_{1}+\kappa \varphi _{m}+\varepsilon \right) \rightarrow -\infty
\end{equation*}%
as $\varepsilon \downarrow 0$, a contradiction.

\item[(2)$\Rightarrow $(3):] Since $\lambda _{2}>0$, we get $m=1$. Let $u\in
H^{2}\left( M\right) $ such that $u$ touches zero somewhere, $\left\Vert
u\right\Vert _{L^{2}}=1$ and $E\left( u\right) =\nu $. We claim $\sharp
u^{-1}\left( 0\right) =1$. Indeed if $\sharp u^{-1}\left( 0\right) >1$, then
by the discussion before Proposition \ref{prop4.1} we know $P\left( u\right)
=\nu u$ and $\nu =\lambda _{1}$. Its eigenfunction $u$ can not touch zero, a
contradiction. The claim follows i.e. $u$ touches $0$ exactly once. Assume $%
u\left( p\right) =0$ and $u>0$ on $M\backslash \left\{ p\right\} $, then%
\begin{equation*}
P\left( u\right) =\nu u+\alpha \delta _{p}.
\end{equation*}%
Hence%
\begin{equation*}
\int_{M}P\left( u\right) G_{L,p}^{-1}d\mu =\nu \int_{M}uG_{L,p}^{-1}d\mu .
\end{equation*}%
Here $G_{L}$ is the Green's function of the conformal Laplacian operator $%
L=-8\Delta +R$. On the other hand it follows from \cite[Proposition 2.1]{HY3}
that%
\begin{equation*}
\int_{M}P\left( u\right) G_{L,p}^{-1}d\mu =\int_{M}uG_{L,p}^{-1}\left\vert
Rc_{G_{L,p}^{4}g}\right\vert _{g}^{2}d\mu .
\end{equation*}%
Combine the two equalities above we get%
\begin{equation*}
\nu \int_{M}uG_{L,p}^{-1}d\mu =\int_{M}uG_{L,p}^{-1}\left\vert
Rc_{G_{L,p}^{4}g}\right\vert _{g}^{2}d\mu .
\end{equation*}%
Hence $\nu \geq 0$.

\item[(3)$\Rightarrow $(1):] If $E\left( u\right) =0$, $u$ is not
identically zero but $u\left( p\right) =0$, then $u=cG_{p}$ (see \cite[%
section 5]{HY1}). Hence $G_{p}\left( p\right) =0$. It follows from \cite[%
Proposition 1.2]{HY3} that $\left( M,g\right) $ is conformal diffeomorphic
to standard $S^{3}$. In this case we know $Y_{4}\left( g\right) >-\infty $
(see \cite{HY1,YZ}). On the other hand if $E\left( u\right) >0$ for any $%
u\in H^{2}\backslash \left\{ 0\right\} $ and $u$ touches zero somewhere,
then the Paneitz operator satisfies condition P and $Y_{4}\left( g\right)
>-\infty $ (see \cite{HY1}).
\end{itemize}
\end{proof}

Indeed the above proof gives us the following

\begin{corollary}
\label{cor4.1}Assume the Yamabe invariant $Y\left( g\right) >0$, $\left(
M,g\right) $ is not conformal diffeomorphic to the standard $S^{3}$ and
there exists a $\widetilde{g}\in \left[ g\right] $ such that $\widetilde{Q}%
\geq 0$ and not identically zero, then the following statements are
equivalent

\begin{enumerate}
\item $Y_{4}\left( g\right) >-\infty $.

\item $\lambda _{2}\left( P\right) >0$.

\item $\nu \left( g\right) >0$ i.e. $P$ satisfies condition P.
\end{enumerate}
\end{corollary}

We remark that Proposition \ref{prop4.1} provides another argument for the
third conclusion of \cite[Theorem 1.1]{HY2}. This approach does not need the
connecting path to Berger's sphere and \cite[Theorem 1.3]{HY1}.

Let $\widetilde{g}=g+th$, for quantities in (\ref{eq4.2}) and (\ref{eq4.3})
we write%
\begin{eqnarray}
\nu \left( g+th,p\right) &=&\nu \left( p\right) +\nu ^{\left( 1\right)
}\left( p,h\right) t+\nu ^{\left( 2\right) }\left( p,h\right) t^{2}+O\left(
t^{3}\right) ,  \label{eq4.9} \\
\alpha \left( g+th,p\right) &=&\alpha \left( p\right) +\alpha ^{\left(
1\right) }\left( p,h\right) t+\alpha ^{\left( 2\right) }\left( p,h\right)
t^{2}+O\left( t^{3}\right) ,  \label{eq4.10} \\
u_{p}\left( g+th,q\right) &=&u_{p}\left( q\right) +u_{p}^{\left( 1\right)
}\left( q,h\right) t+u_{p}^{\left( 2\right) }\left( q,h\right) t^{2}+O\left(
t^{3}\right) .  \label{eq4.11}
\end{eqnarray}%
Hence $u_{p}^{\left( i\right) }\left( p,h\right) =0$ for $i=1,2$. Note
because%
\begin{equation*}
\int_{M}\widetilde{u}_{p}^{2}d\widetilde{\mu }=1,
\end{equation*}%
we have%
\begin{equation}
\int_{M}\left( 2u_{p}\cdot u_{p}^{\left( 1\right) }+\frac{1}{2}%
u_{p}^{2}trh\right) d\mu =0,  \label{eq4.12}
\end{equation}%
and%
\begin{equation}
\int_{M}\left[ 2u_{p}\cdot u_{p}^{\left( 2\right) }+\left( u_{p}^{\left(
1\right) }\right) ^{2}+u_{p}\cdot u_{p}^{\left( 1\right) }trh+\frac{1}{8}%
u_{p}^{2}\left( trh\right) ^{2}-\frac{1}{4}u_{p}^{2}\left\vert h\right\vert
^{2}\right] d\mu =0.  \label{eq4.13}
\end{equation}%
For any smooth function $\varphi $, it follows from (\ref{eq4.2}) that%
\begin{equation*}
\int_{M}\widetilde{u}_{p}\left( \widetilde{P}-\widetilde{\nu }_{p}\right)
\varphi d\widetilde{\mu }=\widetilde{\alpha }_{p}\varphi \left( p\right) .
\end{equation*}%
Hence%
\begin{eqnarray}
&&\alpha _{p}^{\left( 1\right) }\varphi \left( p\right)  \label{eq4.14} \\
&=&\int_{M}\left( u_{p}\left( P_{g,h}^{\left( 1\right) }-\nu _{p}^{\left(
1\right) }\right) \varphi +u_{p}^{\left( 1\right) }\left( P-\nu _{p}\right)
\varphi +\frac{1}{2}u_{p}\left( P-\nu _{p}\right) \varphi \cdot trh\right)
d\mu  \notag
\end{eqnarray}%
and%
\begin{eqnarray}
&&\int_{M}\left[ u_{p}^{\left( 2\right) }\left( P-\nu _{p}\right) \varphi
+u_{p}^{\left( 1\right) }\left( P_{g,h}^{\left( 1\right) }-\nu _{p}^{\left(
1\right) }\right) \varphi +u_{p}\left( P_{g,h}^{\left( 2\right) }-\nu
_{p}^{\left( 2\right) }\right) \varphi \right.  \label{eq4.15} \\
&&+\frac{1}{2}u_{p}^{\left( 1\right) }\left( P-\nu _{p}\right) \varphi \cdot
trh+\frac{1}{2}u_{p}\left( P_{g,h}^{\left( 1\right) }-\nu _{p}^{\left(
1\right) }\right) \varphi \cdot trh  \notag \\
&&\left. +\frac{1}{8}u_{p}\left( P-\nu _{p}\right) \varphi \cdot \left(
trh\right) ^{2}-\frac{1}{4}u_{p}\left( P-\nu _{p}\right) \varphi \cdot
\left\vert h\right\vert ^{2}\right] d\mu  \notag \\
&=&\alpha _{p}^{\left( 2\right) }\varphi \left( p\right) .  \notag
\end{eqnarray}%
By approximation it is also true for $\varphi \in H^{2}\left( M\right) $
too. Hence taking $\varphi =u_{p}$, we get%
\begin{equation}
\nu _{p}^{\left( 1\right) }=\int_{M}u_{p}P_{g,h}^{\left( 1\right) }u_{p}d\mu
.  \label{eq4.16}
\end{equation}%
Similar arguments show%
\begin{eqnarray}
&&\nu _{p}^{\left( 2\right) }  \label{eq4.17} \\
&=&\int_{M}u_{p}P_{g,h}^{\left( 2\right) }u_{p}d\mu +\int_{M}u_{p}^{\left(
1\right) }P_{g,h}^{\left( 1\right) }u_{p}d\mu +\frac{1}{2}%
\int_{M}u_{p}P_{g,h}^{\left( 1\right) }u_{p}\cdot trhd\mu -\frac{1}{4}\nu
_{p}^{\left( 1\right) }\int_{M}u_{p}^{2}trhd\mu .  \notag
\end{eqnarray}

\begin{proposition}
\label{prop4.2}Let $\left( S^{3},g\right) $ be the standard sphere, then for
any $p\in S^{3}$ and smooth symmetric $\left( 0,2\right) $ tensor $h$, we
have%
\begin{equation*}
\nu ^{\left( 1\right) }\left( p,h\right) =0,
\end{equation*}%
and%
\begin{equation*}
\nu ^{\left( 2\right) }\left( p,h\right) =-16II\left( p,p,h\right) .
\end{equation*}
\end{proposition}

Theorem \ref{thm1.2} follows from Proposition \ref{prop3.2} and \ref{prop4.2}%
. By symmetry we can assume $p=N$, then it follows from (\ref{eq4.16}) that%
\begin{eqnarray*}
\nu ^{\left( 1\right) }\left( N,h\right) &=&\int_{S^{3}}u_{N}P_{g,h}^{\left(
1\right) }u_{N}d\mu \\
&=&16\int_{S^{3}}G_{N}P_{g,h}^{\left( 1\right) }G_{N}d\mu \\
&=&0
\end{eqnarray*}%
by the proof of Proposition \ref{prop3.1}.

Next we note that (\ref{eq4.17}) implies%
\begin{eqnarray}
&&\nu ^{\left( 2\right) }\left( N,h\right)  \label{eq4.18} \\
&=&\int_{S^{3}}u_{N}P_{g,h}^{\left( 2\right) }u_{N}d\mu
+\int_{S^{3}}u_{N}^{\left( 1\right) }P_{g,h}^{\left( 1\right) }u_{N}d\mu +%
\frac{1}{2}\int_{S^{3}}u_{N}P_{g,h}^{\left( 1\right) }u_{N}\cdot trhd\mu . 
\notag
\end{eqnarray}%
To compute $u_{N}^{\left( 1\right) }$, we observe that (\ref{eq4.14}) implies%
\begin{equation*}
\int_{S^{3}}\left( u_{N}P_{g,h}^{\left( 1\right) }\varphi +u_{N}^{\left(
1\right) }P\varphi +\frac{1}{2}u_{N}P\varphi \cdot trh\right) d\mu =\alpha
_{N}^{\left( 1\right) }\varphi \left( N\right)
\end{equation*}%
for any $\varphi \in H^{2}\left( S^{3}\right) $. Take $\varphi =u_{p}=4G_{p}$%
, we see%
\begin{equation}
\int_{S^{3}}u_{N}P_{g,h}^{\left( 1\right) }u_{p}d\mu +4u_{N}^{\left(
1\right) }\left( p\right) +2u_{N}\left( p\right) trh\left( p\right) =\alpha
_{N}^{\left( 1\right) }u_{p}\left( N\right) .  \label{eq4.19}
\end{equation}%
Since%
\begin{equation*}
u_{p}\left( N\right) =4G_{p}\left( N\right) =4G_{N}\left( p\right)
=u_{N}\left( p\right) ,
\end{equation*}%
it follows from Lemma \ref{lem2.3} and (\ref{eq4.19}) that%
\begin{equation}
u_{N}^{\left( 1\right) }\left( p\right) =4I_{N}\left( p\right) +\frac{1}{4}%
\alpha _{N}^{\left( 1\right) }u_{N}\left( p\right) .  \label{eq4.20}
\end{equation}%
Hence%
\begin{eqnarray*}
&&\nu ^{\left( 2\right) }\left( N,h\right) \\
&=&16\int_{S^{3}}G_{N}P_{g,h}^{\left( 2\right) }G_{N}d\mu
+16\int_{S^{3}}P_{g,h}^{\left( 1\right) }G_{N}\cdot I_{N}d\mu
+8\int_{S^{3}}P_{g,h}^{\left( 1\right) }G_{N}\cdot G_{N}trhd\mu \\
&=&-16II\left( N,N,h\right) .
\end{eqnarray*}%
Here we have used $\int_{S^{3}}G_{N}P_{g,h}^{\left( 1\right) }G_{N}d\mu =0$
(which follows from the proof of Proposition \ref{prop3.1}) and (\ref{eq3.7}%
). Proposition \ref{prop4.2} follows.

\section{Appendix: Taylor expansion formula for $Q$ curvature of metrics
depending on a parameter}

Assume on a smooth three dimensional Riemannian manifold we have $\widetilde{%
g}=g+th$, then under a local orthonormal frame with respect to $g$, we have

\begin{equation}
\widetilde{g}_{ij}=g_{ij}+th_{ij},  \label{eqa.1}
\end{equation}%
and%
\begin{equation}
\widetilde{g}^{ij}=\delta _{ij}-th_{ij}+t^{2}h_{ij}^{2}+O\left( t^{3}\right)
.  \label{eqa.2}
\end{equation}%
Here $h^{2}$ is the tensor given by%
\begin{equation*}
h_{ij}^{2}=h_{ik}h_{jk}.
\end{equation*}%
The measure associated with $\widetilde{g}$ is given by%
\begin{equation}
d\widetilde{\mu }=\left[ 1+\frac{t}{2}trh+t^{2}\left( \frac{\left(
trh\right) ^{2}}{8}-\frac{\left\vert h\right\vert ^{2}}{4}\right) +O\left(
t^{3}\right) \right] d\mu .  \label{eqa.3}
\end{equation}%
Here%
\begin{equation*}
\left\vert h\right\vert ^{2}=\sum_{ij}\left( h_{ij}\right) ^{2}.
\end{equation*}%
The Christofel symbol satisfies%
\begin{eqnarray}
&&\widetilde{\Gamma }_{ij}^{k}-\Gamma _{ij}^{k}  \label{eqa.4} \\
&=&\frac{t}{2}\left( h_{ikj}+h_{jki}-h_{ijk}\right) -\frac{t^{2}}{2}\left(
h_{i\alpha j}+h_{j\alpha i}-h_{ij\alpha }\right) h_{k\alpha }+O\left(
t^{3}\right) .  \notag
\end{eqnarray}%
Note that the left hand side is a tensor. The Riemann curvature tensor
satisfies%
\begin{eqnarray}
\widetilde{R}_{ijk}^{l} &=&R_{ijk}^{l}+\frac{t}{2}\left(
h_{ilkj}+h_{jkli}+h_{klij}-h_{jlki}-h_{iklj}-h_{klji}\right)  \label{eqa.5}
\\
&&-\frac{t^{2}}{2}\left( h_{i\alpha kj}+h_{jk\alpha i}+h_{k\alpha
ij}-h_{j\alpha ki}-h_{ik\alpha j}-h_{k\alpha ji}\right) h_{l\alpha }  \notag
\\
&&+\frac{t^{2}}{4}\left( h_{i\alpha l}+h_{l\alpha i}-h_{il\alpha }\right)
\left( h_{j\alpha k}+h_{k\alpha j}-h_{jk\alpha }\right)  \notag \\
&&-\frac{t^{2}}{4}\left( h_{i\alpha k}+h_{k\alpha i}-h_{ik\alpha }\right)
\left( h_{j\alpha l}+h_{l\alpha j}-h_{jl\alpha }\right) +O\left(
t^{3}\right) .  \notag
\end{eqnarray}%
In particular, the Ricci tensor is given by%
\begin{eqnarray}
&&\widetilde{Rc}_{ij}  \label{eqa.6} \\
&=&Rc_{ij}+\frac{t}{2}\left( h_{ikjk}+h_{jkik}-\left( trh\right)
_{ij}-\Delta h_{ij}\right)  \notag \\
&&-\frac{t^{2}}{2}\left( h_{ikjl}+h_{jkil}-h_{klji}-h_{ijkl}\right) h_{kl} 
\notag \\
&&+\frac{t^{2}}{4}\left( h_{ikl}+h_{kli}-h_{ilk}\right) \left(
h_{jkl}+h_{klj}-h_{jlk}\right)  \notag \\
&&-\frac{t^{2}}{4}\left( h_{ikj}+h_{jki}-h_{ijk}\right) \left(
2h_{kll}-\left( trh\right) _{k}\right) +O\left( t^{3}\right)  \notag
\end{eqnarray}%
here the tensor $\Delta h$ is given by%
\begin{equation*}
\Delta h_{ij}=h_{ijkk}.
\end{equation*}%
The scalar curvature is given by%
\begin{eqnarray}
&&\widetilde{R}  \label{eqa.7} \\
&=&R+t\left( h_{ijij}-\Delta trh-Rc_{ij}h_{ij}\right)  \notag \\
&&+t^{2}\left( \Delta h_{ij}+\left( trh\right)
_{ij}-h_{ikjk}-h_{ikkj}\right) h_{ij}+\frac{t^{2}}{4}\sum_{ijk}\left(
h_{ikj}+h_{jki}-h_{ijk}\right) ^{2}  \notag \\
&&-\frac{t^{2}}{4}\sum_{i}\left( 2h_{ijj}-\left( trh\right) _{i}\right)
^{2}+t^{2}Rc_{ij}h_{ij}^{2}+O\left( t^{3}\right) .  \notag
\end{eqnarray}%
The Laplacian satisfies%
\begin{eqnarray}
&&\widetilde{\Delta }\varphi  \label{eqa.8} \\
&=&\Delta \varphi -\frac{t}{2}\left( 2h_{ijj}-\left( trh\right) _{i}\right)
\varphi _{i}-th_{ij}\varphi _{ij}  \notag \\
&&+\frac{t^{2}}{2}\left( 2h_{ikk}-\left( trh\right) _{i}\right)
h_{ij}\varphi _{j}+\frac{t^{2}}{2}\left( 2h_{ikj}-h_{ijk}\right)
h_{ij}\varphi _{k}+t^{2}h_{ij}^{2}\varphi _{ij}+O\left( t^{3}\right) . 
\notag
\end{eqnarray}

As a consequence we have%
\begin{eqnarray}
&&\widetilde{R}^{2}  \label{eqa.9} \\
&=&R^{2}+2tR\left( h_{ijij}-\Delta trh-Rc_{ij}h_{ij}\right)  \notag \\
&&+t^{2}\left( h_{ijij}-\Delta trh-Rc_{ij}h_{ij}\right) ^{2}+2t^{2}R\left(
\Delta h_{ij}+\left( trh\right) _{ij}-h_{ikjk}-h_{ikkj}\right) h_{ij}  \notag
\\
&&+\frac{t^{2}R}{2}\sum_{ijk}\left( h_{ikj}+h_{jki}-h_{ijk}\right) ^{2}-%
\frac{t^{2}R}{2}\sum_{i}\left( 2h_{ijj}-\left( trh\right) _{i}\right)
^{2}+2t^{2}R\cdot Rc_{ij}h_{ij}^{2}  \notag \\
&&+O\left( t^{3}\right)  \notag
\end{eqnarray}%
and%
\begin{eqnarray}
&&\left\vert \widetilde{Rc}\right\vert ^{2}  \label{eqa.10} \\
&=&\left\vert Rc\right\vert ^{2}+tRc_{ij}\left( 2h_{ikjk}-\left( trh\right)
_{ij}-\Delta h_{ij}\right) -2tRc_{ij}Rc_{ik}h_{jk}  \notag \\
&&+\frac{t^{2}}{4}\sum_{ij}\left( h_{ikjk}+h_{jkik}-\left( trh\right)
_{ij}-\Delta h_{ij}\right) ^{2}-t^{2}Rc_{ij}h_{kl}\left(
2h_{ikjl}-h_{klij}-h_{ijkl}\right)  \notag \\
&&-2t^{2}Rc_{ij}h_{ik}\left( h_{jlkl}+h_{kljl}-\left( trh\right)
_{jk}-\Delta h_{jk}\right)  \notag \\
&&+\frac{t^{2}}{2}Rc_{ij}\left( h_{ikl}+h_{kli}-h_{ilk}\right) \left(
h_{jkl}+h_{klj}-h_{jlk}\right)  \notag \\
&&-\frac{t^{2}}{2}Rc_{ij}\left( 2h_{ikj}-h_{ijk}\right) \left(
2h_{kll}-\left( trh\right) _{k}\right)
+2t^{2}Rc_{ij}Rc_{ik}h_{jk}^{2}+t^{2}Rc_{ij}Rc_{kl}h_{ik}h_{jl}  \notag \\
&&+O\left( t^{3}\right)  \notag
\end{eqnarray}%
and%
\begin{eqnarray}
&&\widetilde{\Delta }\widetilde{R}  \label{eqa.11} \\
&=&\Delta R+t\Delta \left( h_{ijij}-\Delta trh-Rc_{ij}h_{ij}\right) -\frac{t%
}{2}\left( 2h_{ijj}-\left( trh\right) _{i}\right) R_{i}-th_{ij}R_{ij}  \notag
\\
&&+t^{2}\Delta \left[ \left( \Delta h_{ij}+\left( trh\right)
_{ij}-h_{ikjk}-h_{ikkj}\right) h_{ij}\right] -t^{2}h_{ij}\left(
h_{klkl}-\Delta trh-Rc_{kl}h_{kl}\right) _{ij}  \notag \\
&&+\frac{t^{2}}{4}\Delta \sum_{ijk}\left( h_{ikj}+h_{jki}-h_{ijk}\right)
^{2}-\frac{t^{2}}{4}\Delta \sum_{i}\left( 2h_{ijj}-\left( trh\right)
_{i}\right) ^{2}  \notag \\
&&-\frac{t^{2}}{2}\left( 2h_{ijj}-\left( trh\right) _{i}\right) \left(
h_{klkl}-\Delta trh-Rc_{kl}h_{kl}\right) _{i}+t^{2}\Delta \left(
Rc_{ij}h_{ij}^{2}\right)  \notag \\
&&+\frac{t^{2}}{2}\left( 2h_{ikk}-\left( trh\right) _{i}\right) h_{ij}R_{j}+%
\frac{t^{2}}{2}\left( 2h_{ikj}-h_{ijk}\right)
h_{ij}R_{k}+t^{2}h_{ij}^{2}R_{ij}+O\left( t^{3}\right) .  \notag
\end{eqnarray}%
Recall%
\begin{equation}
Q=-\frac{1}{4}\Delta R-2\left\vert Rc\right\vert ^{2}+\frac{23}{32}R^{2}.
\label{eqa.12}
\end{equation}%
Plug in the formulas above we get%
\begin{eqnarray}
&&\widetilde{Q}  \label{eqa.13} \\
&=&Q-\frac{t}{4}\Delta \left( h_{ijij}-\Delta trh-Rc_{ij}h_{ij}\right)
-2tRc_{ij}\left( 2h_{ikjk}-\left( trh\right) _{ij}-\Delta h_{ij}\right) 
\notag \\
&&+\frac{23t}{16}R\left( h_{ijij}-\Delta trh-Rc_{ij}h_{ij}\right) +\frac{t}{8%
}\left( 2h_{ijj}-\left( trh\right) _{i}\right) R_{i}+\frac{t}{4}%
h_{ij}R_{ij}+4tRc_{ij}Rc_{ik}h_{jk}  \notag \\
&&-\frac{t^{2}}{4}\Delta \left[ \left( \Delta h_{ij}+\left( trh\right)
_{ij}-h_{ikjk}-h_{ikkj}\right) h_{ij}\right] +\frac{t^{2}}{4}h_{ij}\left(
h_{klkl}-\Delta trh-Rc_{kl}h_{kl}\right) _{ij}  \notag \\
&&-\frac{t^{2}}{16}\Delta \sum_{ijk}\left( h_{ikj}+h_{jki}-h_{ijk}\right)
^{2}+\frac{t^{2}}{16}\Delta \sum_{i}\left( 2h_{ijj}-\left( trh\right)
_{i}\right) ^{2}  \notag \\
&&+\frac{t^{2}}{8}\left( 2h_{ijj}-\left( trh\right) _{i}\right) \left(
h_{klkl}-\Delta trh-Rc_{kl}h_{kl}\right) _{i}-\frac{t^{2}}{4}\Delta \left(
Rc_{ij}h_{ij}^{2}\right)  \notag \\
&&-\frac{t^{2}}{2}\sum_{ij}\left( h_{ikjk}+h_{jkik}-\left( trh\right)
_{ij}-\Delta h_{ij}\right) ^{2}+2t^{2}Rc_{ij}h_{kl}\left(
2h_{ikjl}-h_{klij}-h_{ijkl}\right)  \notag \\
&&+4t^{2}Rc_{ij}h_{ik}\left( h_{jlkl}+h_{kljl}-\left( trh\right)
_{jk}-\Delta h_{jk}\right) +\frac{23t^{2}}{32}\left( h_{ijij}-\Delta
trh-Rc_{ij}h_{ij}\right) ^{2}  \notag \\
&&+\frac{23t^{2}}{16}R\left( \Delta h_{ij}+\left( trh\right)
_{ij}-h_{ikjk}-h_{ikkj}\right) h_{ij}-t^{2}Rc_{ij}\left(
h_{ikl}+h_{kli}-h_{ilk}\right) \left( h_{jkl}+h_{klj}-h_{jlk}\right)  \notag
\\
&&+t^{2}Rc_{ij}\left( 2h_{ikj}-h_{ijk}\right) \left( 2h_{kll}-\left(
trh\right) _{k}\right) +\frac{23t^{2}}{64}R\sum_{ijk}\left(
h_{ikj}+h_{jki}-h_{ijk}\right) ^{2}  \notag \\
&&-\frac{23t^{2}}{64}R\sum_{i}\left( 2h_{ijj}-\left( trh\right) _{i}\right)
^{2}-\frac{t^{2}}{8}\left( 2h_{ikk}-\left( trh\right) _{i}\right)
h_{ij}R_{j}-\frac{t^{2}}{8}\left( 2h_{ikj}-h_{ijk}\right) h_{ij}R_{k}  \notag
\\
&&-\frac{t^{2}}{4}%
h_{ij}^{2}R_{ij}-4t^{2}Rc_{ij}Rc_{ik}h_{jk}^{2}-2t^{2}Rc_{ij}Rc_{kl}h_{ik}h_{jl}+%
\frac{23t^{2}}{16}R\cdot Rc_{ij}h_{ij}^{2}+O\left( t^{3}\right) .  \notag
\end{eqnarray}%
Since the Paneitz operator can be written as%
\begin{equation}
P\varphi =\Delta ^{2}\varphi +4Rc_{ij}\varphi _{ij}-\frac{5}{4}R\Delta
\varphi +\frac{3}{4}\nabla R\cdot \nabla \varphi -\frac{1}{2}Q\varphi ,
\label{eqa.14}
\end{equation}%
calculation shows%
\begin{eqnarray}
&&P_{g,h}^{\left( 1\right) }\varphi  \label{eqa.15} \\
&=&-h_{ij}\left( \Delta \varphi \right) _{ij}-\Delta \left( h_{ij}\varphi
_{ij}\right) -\frac{1}{2}\left( 2h_{ijj}-\left( trh\right) _{i}\right)
\left( \Delta \varphi \right) _{i}-\frac{1}{2}\Delta \left[ \left(
2h_{ijj}-\left( trh\right) _{i}\right) \varphi _{i}\right]  \notag \\
&&+2\left( 2h_{ikjk}-\Delta h_{ij}-\left( trh\right) _{ij}\right) \varphi
_{ij}-8Rc_{ij}h_{ik}\varphi _{jk}+\frac{5}{4}Rh_{ij}\varphi _{ij}  \notag \\
&&-\frac{5}{4}\left( h_{ijij}-\Delta trh-Rc_{ij}h_{ij}\right) \Delta \varphi
-2Rc_{ij}\left( 2h_{ikj}-h_{ijk}\right) \varphi _{k}+\frac{5}{8}R\left(
2h_{ijj}-\left( trh\right) _{i}\right) \varphi _{i}  \notag \\
&&+\frac{3}{4}\left( h_{jkjk}-\Delta trh-Rc_{jk}h_{jk}\right) _{i}\varphi
_{i}-\frac{3}{4}h_{ij}R_{i}\varphi _{j}+\frac{1}{8}\Delta \left(
h_{ijij}-\Delta trh-Rc_{ij}h_{ij}\right) \cdot \varphi  \notag \\
&&-\frac{1}{8}h_{ij}R_{ij}\varphi -\frac{1}{16}\left( 2h_{ijj}-\left(
trh\right) _{i}\right) R_{i}\varphi +Rc_{ij}\left( 2h_{ikjk}-\Delta
h_{ij}-\left( trh\right) _{ij}\right) \varphi  \notag \\
&&-2Rc_{ij}Rc_{ik}h_{jk}\varphi -\frac{23}{32}R\left( h_{ijij}-\Delta
trh-Rc_{ij}h_{ij}\right) \varphi .  \notag
\end{eqnarray}%
In particular on standard $S^{3}$ we have%
\begin{eqnarray}
&&P_{g,h}^{\left( 1\right) }\varphi  \label{eqa.16} \\
&=&-h_{ij}\left( \Delta \varphi \right) _{ij}-\Delta \left( h_{ij}\varphi
_{ij}\right) -\frac{1}{2}\left( 2h_{ijj}-\left( trh\right) _{i}\right)
\left( \Delta \varphi \right) _{i}-\frac{1}{2}\Delta \left[ \left(
2h_{ijj}-\left( trh\right) _{i}\right) \varphi _{i}\right]  \notag \\
&&+2\left( 2h_{ikjk}-\Delta h_{ij}-\left( trh\right) _{ij}\right) \varphi
_{ij}-\frac{17}{2}h_{ij}\varphi _{ij}-\frac{5}{4}\left( h_{ijij}-\Delta
trh-2trh\right) \Delta \varphi  \notag \\
&&-\frac{1}{4}\left( 2h_{ijj}-\left( trh\right) _{i}\right) \varphi _{i}+%
\frac{3}{4}\left( h_{jkjk}-\Delta trh-2trh\right) _{i}\varphi _{i}+\frac{1}{8%
}\Delta \left( h_{ijij}-\Delta trh-2trh\right) \varphi  \notag \\
&&-\frac{5}{16}\left( h_{ijij}-\Delta trh-2trh\right) \varphi .  \notag
\end{eqnarray}

\subsection*{Acknowledgement}

The research of Yang is supported by NSF grant 1104536.

\end{document}